\documentclass{article}

\usepackage[T1]{fontenc}
\usepackage{enumerate, amsmath, amsfonts, amssymb, amsthm, dsfont, mathrsfs, wasysym, graphics, graphicx, xcolor, url, hyperref, hypcap, xargs, multicol, pdflscape, multirow, hvfloat, array, ae, aecompl, pifont, mathtools, a4wide, float, blkarray, overpic, nicefrac, stmaryrd, anyfontsize, yfonts, fontawesome}
\usepackage{pdflscape}
\usepackage{xargs, bbm, enumerate, paralist}
\usepackage[noabbrev,capitalise]{cleveref}
\usepackage[normalem]{ulem}
\usepackage{marginnote}
\usepackage{mathrsfs}
\usepackage{enumitem}
\usepackage{svg}
\usepackage{cleveref}
\usepackage{mathdots}

\hypersetup{colorlinks=true, citecolor=darkblue, linkcolor=darkblue}
\usepackage[all]{xy}
\usepackage{tikz}
\usepackage{tikz-cd}
\usetikzlibrary{trees, decorations, decorations.pathmorphing, decorations.markings, decorations.shapes, shapes, arrows, matrix, calc, fit, intersections, patterns, angles}
\graphicspath{{figures/}{figures/diagonals/}{figures/walks/}{figures/tubes/}{figures/blocks/}}
\makeatletter\def\input@path{{figures/}}\makeatother
\usepackage{caption}
\usepackage[export]{adjustbox}

\usepackage{paralist}

\DeclareFontEncoding{LY}{}{}
\DeclareFontSubstitution{LY}{yfrak}{m}{n}
\DeclareFontEncoding{LYG}{}{}
\DeclareFontSubstitution{LYG}{ygoth}{m}{n}
\DeclareFontFamily{LYG}{ygoth}{}
\DeclareFontShape{LYG}{ygoth}{m}{n}{<->ygoth}{}
\DeclareFontFamily{LY}{yfrak}{}
\DeclareFontShape{LY}{yfrak}{m}{n}{<->yfrak}{}
\DeclareFontFamily{LY}{ysmfrak}{}
\DeclareFontShape{LY}{ysmfrak}{m}{n}{<->ysmfrak}{}
\DeclareFontFamily{LY}{yswab}{}
\DeclareFontShape{LY}{yswab}{m}{n}{<->yswab}{}

\newtheorem{thmUniv}{Theorem}

\newtheorem{conjUniv}{Conjecture}

\newtheorem{theorem}{Theorem}[section]
\newtheorem{corollary}[theorem]{Corollary}
\newtheorem{proposition}[theorem]{Proposition}
\newtheorem{lemma}[theorem]{Lemma}
\newtheorem{conjecture}[theorem]{Conjecture}
\newtheorem*{theorem*}{Theorem}%

\theoremstyle{definition}
\newtheorem{definition}[theorem]{Definition}
\newtheorem{example}[theorem]{Example}
\newtheorem{remark}[theorem]{Remark}
\newtheorem{question}{Question}

\newtheorem{Question}{Open Question}

\AddToHook{env/theorem/begin}{\crefalias{section}{theorem}}
\AddToHook{env/proposition/begin}{\crefalias{theorem}{proposition}}
\AddToHook{env/corollary/begin}{\crefalias{theorem}{corollary}}
\AddToHook{env/lemma/begin}{\crefalias{theorem}{lemma}}
\AddToHook{env/conjecture/begin}{\crefalias{theorem}{conjecture}}
\AddToHook{env/definition/begin}{\crefalias{theorem}{definition}}
\AddToHook{env/example/begin}{\crefalias{theorem}{example}}
\AddToHook{env/remark/begin}{\crefalias{theorem}{remark}}
\AddToHook{env/observation/begin}{\crefalias{theorem}{observation}}
\AddToHook{env/question/begin}{\crefalias{theorem}{question}}
\AddToHook{env/notation/begin}{\crefalias{theorem}{notation}}
\AddToHook{env/assumption/begin}{\crefalias{theorem}{assumption}}
\AddToHook{env/convention/begin}{\crefalias{theorem}{convention}}

\crefname{thmUniv}{theorem}{theorems}
\Crefname{thmUniv}{Theorem}{Theorems}
\crefname{conjUniv}{conjecture}{conjectures}
\Crefname{conjUniv}{Conjecture}{Conjectures}

\crefname{equation}{Equation}{Equations}

\newcommand{\R}{\mathbb{R}} %
\newcommand{\Q}{\mathbb{Q}} %
\newcommand{\N}{\mathbb{N}} %
\newcommand{\Z}{\mathbb{Z}} %
\newcommand{\cL}{\mathcal{L}} %
\renewcommand{\c}[1]{{\mathcal{#1}}} %
\renewcommand{\b}[1]{{\boldsymbol{#1}}} %
\renewcommand{\epsilon}{\varepsilon} %
\newcommand{\polytope}[1]{\mathsf{#1}}

\newcommand\pol{\polytope{P}}

\newcommand{\simplex}{\polytope{\Delta}} %
\DeclareMathOperator{\conv}{conv} %

\DeclareMathOperator{\vol}{Vol}

\newcommand{\ie}{\textit{i.e.,}~} %
\newcommand{\eg}{\textit{e.g.,}~} %
\newcommand{\aka}{\textit{a.k.a.}~} %
\definecolor{darkblue}{rgb}{0,0,0.7} %
\definecolor{green}{RGB}{57,181,74} %
\definecolor{violet}{RGB}{147,39,143} %
\newcommand{\darkblue}{\color{darkblue}} %
\newcommand{\defn}[1]{\textsl{\darkblue #1}} %
\newcommand{\mathdefn}[1]{{\darkblue #1}} %
\newcommand{\asc}{\text{asc}}

\newcommand{\Asc}{\text{Asc}}
\newcommand{\Pyr}{\polytope{Pyr}}
\newcommand{\Sn}{\mathfrak{S}_n}

\newcommand{\Is}{I_n^{\b s}}
\newcommand{\Pn}{\mathcal{P}_n^{\b s}}
\newcommandx{\Epoly}[3][1=\b s, 2=n, 3=t]{E_{#2}^{#1}(#3)}

\usepackage{todonotes}

\AtEndDocument{\bigskip{\footnotesize%
 \textsc{Universität Osnabrück, Germany} \\
 Martina Juhnke: \texttt{martina.juhnke@uni-osnabrueck.de} \\ 
Germain Poullot: \texttt{germain.poullot@uni-osnabrueck.de} \\
Jhon  B. Caicedo: \texttt{jhon.bladimir.caicedo.portilla@uni-osnabrueck.de}\par }}

\begin{document}%

\title{Ehrhart non-positivity and unimodular triangulations for classes of $\b s$-lecture hall simplices}
\author{Jhon B. Caicedo, Martina Juhnke and Germain Poullot}
\date{}

\maketitle

\begin{abstract}

Counting lattice points and triangulating polytopes are prominent subjects in discrete geometry, yet proving Ehrhart positivity or existence of unimodular triangulations remain of utmost difficulty in general,  even for simplices.
We study these problems for classes of $\b s$-lecture hall simplices.%

Inspired by a question of Olsen, we present a new \emph{natural} class of sequences $\b s$ for which the $\b s$-lecture hall simplices are not Ehrhart positive, by explicitly estimating a negative coefficient.
Meanwhile, motivated by a conjecture of Hibi, Olsen and Tsuchiya, we extend the previously known classes of sequences $\b s$ for which the $\b s$-lecture hall simplex admits a flag, regular and unimodular triangulation.
The triangulations we construct are explicit.

\end{abstract}

\renewcommand{\baselinestretch}{0.86}\normalsize
\tableofcontents
\renewcommand{\baselinestretch}{1.0}\normalsize

\section*{Acknowledgement}
Jhon B. Caicedo was supported by DFG grant JU 3097/4-1, he is also thankful to Katharina Jochemko and Torben Donzelmann for insightful mathematical conversations.

Germain Poullot wants to thank the two other authors for asking what seemed a computational question and turned out to be an exciting paper.

We thank the anonymous referee for his or her useful editing suggestions.

\section{Introduction}

Lecture hall polytopes form an important family of lattice polytopes that have attracted significant attention in recent years due to their rich combinatorial structure and intriguing geometric behavior.
Given a sequence of positive integers $\b s = (s_1, \ldots, s_n)$, the  $\b s$-lecture hall simplex, denoted by \defn{$\Pn$}, is defined as the set of points $\b x = (x_1,\ldots,x_n)\in \R^n$ satisfying the inequalities

$$0\leq \frac{x_1}{s_1}\leq \cdots \leq \frac{x_n}{s_n}\leq 1.$$

Although $\b s$-lecture hall simplices have been extensively studied, they continue to raise fundamental questions in very active areas of mathematics, including Ehrhart theory and polyhedral geometry \cite{Olsen-question,GUSTAFSSON2020107169, FlorianAndOlsen}.  
This work is motivated by the two central problems on $\b s$-lecture hall simplices: their Ehrhart (non)positivity, and the existence of unimodular triangulations.

\paragraph{Ehrhart (non) positivity}
In \cite{Olsen-question}, Olsen asked the following question (see Question 6.8 therein): 
\begin{question}
    Under which conditions on the sequence $\b s$ is the $\b s$-lecture hall simplex $\mathcal{P}_n^{\b s}$ Ehrhart positive or \emph{not} Ehrhart positive?
\end{question}

Some partial results regarding this question have been established. 
It is well-known, see Savage and Schuster \cite[Corollary~1]{SAVAGE2012850}, that if $\b s = (1, 2, \ldots, n)$, then the coefficients of the Ehrhart polynomial of $\Pn$ are the binomial coefficients (\ie its Ehrhart polynomial is $\c L_{\Pn}(t) = (t+1)^n$), hence they are positive.
They also proved in \cite[Theorem 13]{SAVAGE2012850} that for the sequence defined by $s_{2i} = 4i$ and $s_{2i - 1} = 2i - 1$, \ie $\b s = (1, 4, 3, 8, 5, 12, \ldots)$, we get $\mathcal{L}_{\mathcal{P}_n^{\b s}}(t) = (t+1)^{\lceil n/2 \rceil} (2t + 1)^{\lfloor n/2 \rfloor}$, whose coefficients are also positive.
Positivity results have been shown for similar sequences.

On the opposite, Liu and Solus \cite[Theorem 4.3]{LiuAndSolus2019} specified conditions on $a$, $b$ and the number of $1$s in the sequence $\b s = (1,\dots,1,a,1,\dots,1,b,1,\dots,1)$ such that its $\b s$-lecture hall simplex is not Ehrhart positive (\ie there is at least one negative coefficient in its Ehrhart polynomial $\cL_{\Pn}(t)$).

We exhibit a new and somewhat unexpected sequence $\b s$ for which the lecture hall simplex $\Pn$ fails to be Ehrhart positive.
Specifically, we prove the following theorem.

\begin{thmUniv}[{\Cref{The:Negativity}}]\label{thm:A}
For any $n\geq 5$, and any $a$ big enough, the simplex $\Pn$ is not Ehrhart positive for the sequence $\b s = (a,\dots, a, a + 1)$ of length $n$.
More precisely, as $a\to \infty$,
$$[t^{n-4}]\cL_{\Pn}(t)\sim -\,\frac{1}{720(n-4)!}\, a^{n-1}$$
where $[t^{n-4}]\cL_{\Pn}(t)$ denotes the coefficient of $\cL_{\Pn}(t)$ in front of $t^{n-4}$.
\end{thmUniv}

This result came as a surprise to us since, for $\b a = (a, \ldots, a)$, the  simplex $\mathcal{P}_n^{\b a}$ is unimodularly equivalent to the $a$\textsuperscript{th} dilation of the  standard simplex, which is well-known to be Ehrhart positive (its Ehrhart polynomial is $\c L_{\c P_n^{\b a}}(t) = \binom{at+n}{n+1}$).
This further emphasizes the subtle combinatorial and geometric complexity of $\b s$-lecture hall simplices.

\paragraph{Unimodular (flag, regular) triangulations}
The second problem is the following conjecture by Hibi, Olsen, and Tsuchiya:
\begin{conjUniv}[{\cite[Conjecture~5.2]{hibi2016}}]
For any sequence of positive integers $\b s$, the polytope $\mathcal{P}_n^{\b s}$ admits a unimodular triangulation.
\end{conjUniv}

Evidence in support of this conjecture has been provided in several cases.
The first was given by Hibi, Olsen, and Tsuchiya using the concept of \emph{chimney polytopes} (see \cite[Section~3]{hibi2016}).
They considered a sequence $\b s = (s_1, \ldots, s_n)$ of positive integers such that for all $i$, one has $s_{i+1} = k_i\cdot s_i$ or $s_{i+1} = \frac{s_i}{k_i}$ for  integers $k_i \geq 1$. 
In this case, the $\b s$-lecture hall simplex admits a unimodular triangulation. (Even though Hibi, Olsen and Tsuchiya only mention unimodularity  explicitly, their triangulation is also flag and regular, we have no doubt that these authors knew it.)
Moreover, Br\"and\'en and Solus (see \cite[Corollary~3.5]{SolusAndBradden}) showed that for any sequence $\b s = (s_1, \ldots, s_n)$ satisfying $0\leq s_{i+1} - s_i \leq 1$ the corresponding $\b s$-lecture hall simplex admits a flag, regular, and unimodular triangulation.
Yet, their construction is not explicit as it relies heavily on Gr\"obner bases.

Building on these developments, we propose in \Cref{ssec:PrelimPointExtension} a unifying framework that will allow us in \Cref{Section:Regular_Unimodular} to generalize the results of Br\"and\'en and Solus by leveraging the construction of Hibi, Olsen and Tsuchiya. Our construction has the asset of keeping the triangulation explicit, using purely combinatorial and geometric arguments.
Explicitly, we have the following (see \Cref{ssec:PrelimTriangulations} for the  definitions of flag, regular, unimodular):
\newpage

\begin{thmUniv}[{\Cref{cor:extended,coro:join,Remark:dilation}}]\label{thm:B}
Let $\b s = (s_1,\dots, s_n)\in \N^n$. 
The following two cases ensures that $\mathcal{P}_{n}^{\b s}$ has a flag, regular, unimodular triangulation:
\begin{compactenum}
\item[(a)] 
If there exists $1\leq m\leq n$ such that
\begin{compactenum}
\item[(1)] for all $1\leq i\leq m-1$ either $s_i = 1$ or there is $k_i\geq 1$, $\epsilon_i\in\{0,1\}$ with $s_{i+1} =  \frac{s_{i}-\epsilon_i}{k_i}$; and,
\item[(2)] for all $m\leq i\leq n-1$ either $s_{i+1} = 1$ or there is $k_i\geq 1$, $\epsilon_i\in\{0,1\}$ with $s_{i+1} = k_i\cdot s_{i} + \epsilon_i$;
\end{compactenum}
then $\mathcal{P}_{n}^{\b s}$ has a flag, regular, unimodular triangulation.

\item[(b)] If $\b s = (\lambda_1\b s_1, 1, \lambda_2\b s_2, 1, \dots, 1, \lambda_r\b s_r)$ where for each $\b s_j\in \N^{n_j}$ the polytope $\c P_{n_j}^{\b s_j}$ admits a flag, regular, unimodular triangulation (\eg if $\b s_j$ satisfies (a)), and $\lambda_j\b s_j = (\lambda_j\cdot\b s_{j, 1}, \dots, \lambda_j\cdot\b s_{j, n_j})$ for some integer $\lambda_j\geq 0$, then $\c P_n^{\b s}$ admits a flag, regular, unimodular triangulation.
\end{compactenum}
\end{thmUniv}

To conclude \Cref{Section:Regular_Unimodular}, we prove the existence of flag, regular, unimodular triangulations of $\Pn$ for the sequences $\b s = (a, \dots, a, a+1)$ and $\b s = (1, \dots, 1, a, 1, \dots, 1, b, 1, \dots, 1)$  for which the $\b s$-lecture hall simplex is known to be Ehrhart non-positive (with conditions on $a$, $b$, and the number of consecutive $1$s, see \Cref{thm:positivity} and \cite[Theorem~4.3]{LiuAndSolus2019}).
Finally, motivated by this new construction, we propose a strengthening of the  previous conjecture, namely:

\begin{conjUniv}
For any sequence of positive integers $\b s$, the polytope $\mathcal{P}_n^{\b s}$ admits a triangulation which is flag, regular, and unimodular.
\end{conjUniv}

After recalling some background knowledge on Ehrhart theory, triangulations, and $\b s$-lecture hall simplices in \Cref{section:preliminaries}, we explain how to triangulate a one-point extension of a polytope in \Cref{ssec:PrelimPointExtension}.
We prove \Cref{thm:A} in \Cref{Section:Ehrhart_nonpositive}, while \Cref{Section:Regular_Unimodular} is devoted to \Cref{thm:B}.

\section{Preliminaries}{\label{section:preliminaries}}
In this section, we provide necessary background on Ehrhart theory, as well as $\b s$-lecture hall simplices, and their relation to $\b s$-Eulerian polynomials. We refer to \cite{SAVAGE2012850,beck2015computing} for more details.

\subsection{Ehrhart polynomial and $h^\ast$-vector}\label{Def:Ehrhart_poly}

A \defn{$d$-dimensional lattice polytope} $\pol\subset \R^n$ is the convex hull of finitely many points in $\Z^n$ whose affine span is $d$-dimensional. 
For a positive integer $t$, we let $\mathdefn{t\polytope{P}} \coloneqq \{t\b p ~:~ \b p\in \polytope{P}\}$ be the $t$\textsuperscript{th} dilation of $\polytope{P}$. 
The \defn{lattice point enumerator $\mathcal{L}_{\pol}$} (\aka \defn{Ehrhart polynomial}) of $\pol$ counts the numbers of lattice points in the (integer) dilates of $\pol$. More precisely, for $t \in \N$, let
\[
\mathcal{L}_{\pol}(t) \coloneqq \bigl|t\pol \cap \Z^n\bigr|.
\]

A fundamental theorem going back to Ehrhart \cite{Ehrhart_poly} ensures that, for a $d$-dimensional lattice polytope $\pol$, the counting function $\cL_{\pol}(t)$ is a polynomial of degree $d$, with leading term equal to the relative volume of $\pol$.
The polytope $\polytope{P}$ is called \defn{Ehrhart positive} if all coefficients of $\c L_{\pol}(t)$ (written in the standard monomial basis) are non-negative.
Determining if a polytope is Ehrhart positive or not is a rich and active topic in (lattice) polytope theory (see \cite{Liu2019}).
The generating function of $\cL_{\pol}(t)$, referred to as \defn{Ehrhart series} of $\pol$, is given as the following rational function:
\[
\sum_{t\geq 0}\mathcal{L}_{\pol}(t)z^{t} = \frac{1}{(1-z)^{d+1}}\bigl(h_0^\ast + h_1^\ast z + \dots + h_{d}^\ast\,z^{d}\bigr).
\]

The polynomial $h_0^\ast + h_1^\ast z + \dots + h_d^\ast z^{d}$ is called the \defn{$h^\ast$-polynomial} of $\pol$, and the vector $(h_0^\ast,\ldots,h^\ast_d)$ is called the \defn{$h^*$-vector} of $\pol$. 
Stanley’s non-negativity theorem \cite[Theorem~2.1]{stanley1980decompositions} ensures that the $h^\ast$-vector has only non-negative integer entries.
By definition, the Ehrhart polynomial and the $h^\ast$-vector carry the same information.
Especially, the former can be retrieved from the latter:

\begin{lemma}[{\cite[Lemma 3.14]{beck2015computing}}]\label{Lemma:Beck_book}
For a $d$-dimensional lattice polytope $\pol$,
$$\mathcal{L}_{\pol}(t) = h^\ast_d\binom{t}{d} ~ + \cdots +~ h_1^\ast\binom{t + d - 1}{d} + h_0^\ast\binom{t + d}{d}.$$

\end{lemma}
\subsection{Triangulations}\label{ssec:PrelimTriangulations}
A common method to compute $h^\ast$-vectors is \emph{via} triangulations. 
A \defn{triangulation} of a $d$-dimensional lattice polytope $\pol$ is a collection of $d$-dimensional lattice simplices $\c T$ such that:
(1) the intersection of any two simplices of $\c T$ is a common face of both; and (2) the union of all simplices of $\mathcal{T}$ is $\pol$.

A triangulation is \defn{unimodular} if all its simplices have normalized volume $1$ (\ie their relative Euclidean volume is $\frac{1}{d!}$). In other words, a unimodular triangulation is a triangulation into lattice simplices of smallest possible volume. 
 \Cref{fig:Example_1} shows two unimodular and one non-unimodular triangulations.

A triangulation is \defn{regular} if it arises as the projection of the lower-hull of a lifting of its vertices.
More precisely, denoting $\b p_1, \dots, \b p_r$ the vertices of $\c T$, we say that $\c T$ is regular if there exist heights $\omega_1, \dots, \omega_r\in \R$ such that that the lower facets\footnote{A facet is a lower facet if its outer normal vector has a (strictly) negative last coordinate.} of $\conv\bigl((\b p_i, \omega_i) ~:~ i\in [r]\bigr)$ are the simplices $\conv\bigl((\b p_i, \omega_i) ~:~ i\in X\bigr)$ for all $X\subseteq [r]$ such that $\conv(\b p_i ~: ~i\in X) \in \c T$.

Finally, we say that a triangulation $\c T$ is \defn{flag} if the minimal nonfaces of its asscociated simplicial complexes have cardinality $2$. In other words, $\c T$ is completely determined by its $1$-skeleton.%

\subsection{$\b s$-lecture hall simplices and $\b s$-Eulerian polynomials}
For a sequence $\b s = (s_1,\dots, s_n)$ of positive integers, the \defn{$\b s$-lecture hall simplex} $\mathcal{P}_n^{\b s}$ is defined as
\[
\mathcal{P}_n^{\b s}\coloneqq \left\{\b x\in \R^n ~:~ 0 \leq \frac{x_1}{s_1} \leq \dots \leq \frac{x_n}{s_n}\leq 1\right\}.
\]
Alternatively, it can be expressed as the convex hull of its vertices:
\[
\mathcal{P}_n^{\b s} = \text{conv}\bigl\{(0, \dots , 0),\, (0, \dots , 0, s_n),\, \dots ,\, (s_1,\dots , s_n)\bigr\}.
\]

The $\b s$-lecture hall simplices have been originally introduced by Savage and Schuster \cite{SAVAGE2012850} and have recently attracted a lot of attention, see e.g.,  \cite{Olsen-question,SolusAndBradden, GUSTAFSSON2020107169, FlorianAndOlsen}.
In \cite{SAVAGE2012850}, the authors show that the $h^\ast$-polynomial of $\Pn$ coincides with the following \defn{$\b s$-Eulerian polynomial}.

\begin{definition}\label{def:sInversionSequences}
For a sequence $\b s = (s_1, \dots, s_n)\in\N^n$, the \defn{$\b s$-inversion sequences} are the lattice points of the half-open parallelotope $[0 , s_1)\times \cdots  \times [0 , s_n)$; equivalently, they are elements of
\[
\mathdefn{\Is} \coloneqq \bigl\{(e_1,\ldots , e_n)\in \N^n~:~0 \leq e_i < s_i \text{ for all } i\in[n]\bigr\}.
\]

A sequence $\b e\in \Is$ has an \defn{ascent} at position $0\leq i< n$ if $ \frac{e_i}{s_i} < \frac{e_{i+1}}{s_{i+1}}$ (where $e_0 = s_0 = 1$).
The set of ascents of $\b e$ is $\mathdefn{\Asc(\b e)} \coloneq \left\{0\leq i \leq n-1~ : ~ \frac{e_i}{s_i} < \frac{e_{i+1}}{s_{i+1}}\right\}$, with $\mathdefn{\asc(\b e)} \coloneqq |\Asc(\b e)|$. 
\end{definition}

\begin{definition}\label{def:SEulereianPolynomials}
For a sequence $\b s$ of positive integers, the associated \defn{$\b s$-Eulerian polynomial} is
$$\Epoly = \sum_{\b e\in I_n^{\b s}}t^{\asc(\b e)}\, .$$
\end{definition}

According to the convention of \Cref{def:sInversionSequences}, we have $E_{1}^{(s_1)}(t) = 1 + (s_1 - 1)\, t$, for $n = 1$.

\begin{example}\label{remark:Eulerian_numbers}
For $\b s_{\text{nat}} = (1,\ldots,n)$, the $\b s_{\text{nat}}$-Eulerian polynomial coincides with the classical \defn{Eulerian polynomial} (see \cite[Lemma 1]{SAVAGE2012850}), \ie
\[
\Epoly[\b s_{\text{nat}}] = \sum_{k = 0}^{n - 1} A(n,k)\, t^k,
\]
where the \defn{Eulerian number $A(n, k)$} is the number of permutations on $n$ elements with $k$ ``usual ascents'' (a permutation $\sigma\in\Sn$ has a usual ascent at position $i$ if $\sigma(i)<\sigma(i+1)$).
\end{example}

Classical Eulerian polynomials are known to be real-rooted, and consequently log-concave and unimodal. 
The same properties have been shown more generally for any $\b s$-Eulerian polynomial, see \cite{Savagerealroots}. 
Furthermore, 
according to \cite[Theorem 5]{SAVAGE2012850}, the $\b s$-lecture simplex and the $\b s$-Eulerian polynomial are related via the $h^*$-vector:
\begin{equation}\label{eq:Eulerian-h-star}
  h^\ast_{\Pn}(t)=  E_{n}^{\b s}(t).
\end{equation}
Consequently, the $h^\ast$-polynomial of any $\b s$-lecture hall simplex is real-rooted and, in particular, their $h^\ast$-vectors are unimodal. Unfortunately, as with most real-rootedness proofs, the result from \cite{Savagerealroots} does not offer an algebraic or geometric explanation for the unimodality of these vectors. 

\begin{example}
Let $\b s = (1,2,3)$.
Then $I_3^{\b s} = \{ (0,0,0),\, (0,0,1),\, (0,0,2),\, (0,1,0),\, (0,1,1),\, (0,1,2) \}$.
The $h^\ast$-polynomial of $\mathcal{P}_3^{\b s}$ is given by $h^\ast_{\mathcal{P}_3^{\b s}}(t) = 1 + 4t + t^2$; while its Ehrhart polynomial is given by $\mathcal{L}_{\mathcal{P}_3^{\b s}}(t) = t^3 + 3t^2 + 3t + 1$.
\Cref{fig:Example_1} displays triangulations of $\c P_{3}^{\b s}$ with the final vertex $\b v_3 = (1,2,3)$ omitted, in order to better visualize the structure of its triangulation.  
Since this vertex contributes only a single lattice point (which gives rise to a pyramid at height $1$ over the drawing of \Cref{fig:Example_1}), the (non)unimodularity of the triangulations is preserved.

\begin{figure}[h]
    \centering
    \newcommand{\drawFilteredPoints}{
  \foreach \x in {0,...,2}{
    \foreach \y in {0,...,3}{
      \pgfmathsetmacro{\limy}{1.5*\x}
      \pgfmathparse{\y >= \limy}
      \ifnum\pgfmathresult=1
        \draw (\x,\y) node{$\bullet$};
      \fi
    }
  }
}

\begin{tikzpicture}[scale = 1]

\begin{scope}[shift={(-1, 0)}]
    \draw[-](-0.25 , 0)--(2.5 , 0);
    \draw[-](0,-0.25 )--(0 , 3.5 );

    \draw[very thick, black] (0, 0) -- (0, 3) -- (2,3) -- cycle;

    \draw[dashed, fill=red!30, opacity=0.6] (0, 1) -- (1,2) -- (0,2) -- cycle;
    \draw[dashed, fill=blue!30, opacity=0.6] (1, 2) -- (2,3) -- (1,3) -- cycle;
    \draw[dashed, fill=green!30, opacity=0.6] (0, 2) -- (1,2) -- (1,3) -- cycle;
    \draw[dashed, fill=violet!30, opacity=0.6] (0, 2) -- (0,3) -- (1,3) -- cycle;
    \draw[dashed, fill=yellow!30, opacity=0.6] (0, 0) -- (0,1) -- (1,2) -- cycle;
    \draw[dashed, fill=magenta!30, opacity=0.6] (0, 0) -- (1,2) -- (2,3) -- cycle;

    \draw (0, 0) node[below left]{$0$};
    \foreach \x in {1,...,2}{\draw (\x , 0) node[below]{$\x $};}
    \foreach \y in {1,...,3}{\draw (0, \y) node[left]{$\y $};}

    \drawFilteredPoints
\end{scope}

\begin{scope}[shift={(4, 0)}]
    \draw[-](-0.25 , 0)--(2.5 , 0);
    \draw[-](0,-0.25 )--(0 , 3.5 );

    \draw[very thick, black] (0, 0) -- (0, 3) -- (2,3) -- cycle;

    \draw[dashed, fill=red!30, opacity=0.6] (0, 1) -- (1,2) -- (0,2) -- cycle;
    \draw[dashed, fill=blue!30, opacity=0.6] (1, 2) -- (2,3) -- (1,3) -- cycle;
    \draw[dashed, fill=green!30, opacity=0.6] (0, 2) -- (1,2) -- (0,3) -- cycle;
    \draw[dashed, fill=violet!30, opacity=0.6] (0, 3) -- (1,2) -- (1,3) -- cycle;
    \draw[dashed, fill=yellow!30, opacity=0.6] (0, 0) -- (0,1) -- (1,2) -- cycle;
    \draw[dashed, fill=magenta!30, opacity=0.6] (0, 0) -- (1,2) -- (2,3) -- cycle;

    \draw (0, 0) node[below left]{$0$};
    \foreach \x in {1,...,2}{\draw (\x , 0) node[below]{$\x $};}
    \foreach \y in {1,...,3}{\draw (0, \y) node[left]{$\y $};}

    \drawFilteredPoints
\end{scope}

\begin{scope}[shift={(9, 0)}]
    \draw[-](-0.25 , 0)--(2.5 , 0);
    \draw[-](0,-0.25 )--(0 , 3.5 );

    \draw[very thick, black] (0, 0) -- (0, 3) -- (2,3) -- cycle;

    \draw[dashed, fill=blue!30, opacity=0.6] (1, 2) -- (2,3) -- (0,3) -- cycle;
    \draw[dashed, fill=green!30, opacity=0.6] (0, 1) -- (1,2) -- (0,3) -- cycle;
    \draw[dashed, fill=yellow!30, opacity=0.6] (0, 0) -- (0,1) -- (1,2) -- cycle;
    \draw[dashed, fill=magenta!30, opacity=0.6] (0, 0) -- (1,2) -- (2,3) -- cycle;

    \draw (0, 0) node[below left]{$0$};
    \foreach \x in {1,...,2}{\draw (\x , 0) node[below]{$\x $};}
    \foreach \y in {1,...,3}{\draw (0, \y) node[left]{$\y $};}

    \drawFilteredPoints
\end{scope}

\end{tikzpicture}

\vspace{-0.8cm}
    \caption{(Left \& Middle) Two different unimodular triangulations of $\c P_3^{\b s}$ without its vertex  $\b v_3 = (1, 2, 3)$.
    (Right) A non-unimodular triangulation for $\c P_3^{\b s}$ without its vertex  $\b v_3 = (1, 2, 3)$.}
    \label{fig:Example_1}
\end{figure}
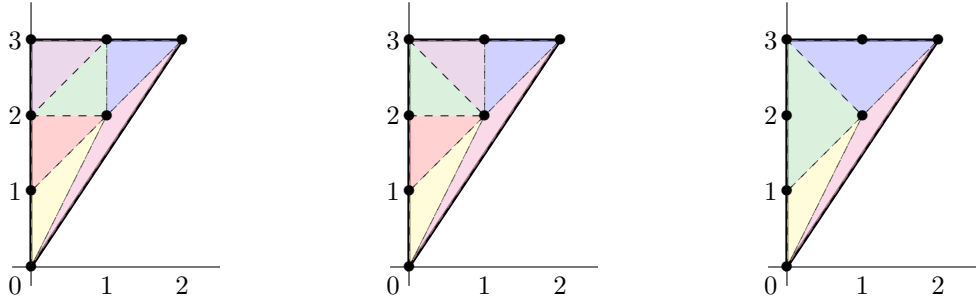
\end{example}

\subsection{Triangulating one-point extensions}\label{ssec:PrelimPointExtension}

Let $\pol\subseteq \R^m$ be a $d$-dimensional lattice polytope, defined as the convex hull of the vertices $\b v_0, \dots, \b v_{n}$ and let $\b x\in\R^{m}$ be a point not contained in the affine hull of $\pol$. 
 The \defn{pyramid  $\Pyr(\pol,\b x)$ over $\pol$} (with apex $\b x$) is the convex hull of the  $\pol$ and $\b x$, \ie $\conv(\pol\cup\{\b x\})$. We will also write $\Pyr\pol$ if the apex is clear from the context. 
 
It is well-known (see \cite[page 36]{beck2015computing}) that, if $\b x $ is at lattice distance $1$ from the affine hull of $\pol$, then the number of integers points contained in the dilation $t\cdot\Pyr(\pol)$ is given by
\begin{equation}\label{Ehrhart: Pyramid}
\mathcal{L}_{\Pyr(\pol)}(t) = 1 + \mathcal{L}_{\pol}(1) + \cdots + \mathcal{L}_{\pol}(t) = \sum_{j = 0}^{t} \mathcal{L}_{\pol}(j).
\end{equation}

In order to also deal with the case where $\b x$ lies in the affine hull of $\pol$, we will apply the pyramid idea not only to $\pol$ but rather to some facets $\polytope{F}$ of~$\pol$. 
More precisely, we define:
\begin{definition}
For $\pol$ a lattice polytope of dimension and a lattice point $\b x\notin \pol$:
\begin{compactenum}
    \item The \defn{one-point extension} of $\pol$ by $\b x$ is the polytope $\conv(\pol\cup\{\b x\})$;
    \item A face $\polytope{F}$ of $\pol$ is \defn{visible from $\b x$} if for all $\b y\in \polytope{F}$, we have $[\b x, \b y]\cap \pol = \{\b y\}$.    
    We denote by \defn{$\c F_{\b x}$} the collection of facets of $\pol$ visible from $\b x$.
\end{compactenum}
\end{definition}

\begin{theorem}\label{thm:One_point_triangulation}
Let $\pol$ a $d$-dimensional lattice polytope.
Then we have the following

\begin{compactenum}
\item\label{item:thm:One_point_triangulation1} Let $\c T$ be a triangulation of $\pol$ and let $\b x\notin \pol$.  
Let $\c T\left|_{\c F_{\b x}}\right. \coloneqq \{\simplex\cap\polytope{F} ~:~ \simplex\in \c T,\, \polytope{F}\in \c F_{\b x}\}$.
Then $$\widetilde{\c T} \coloneqq \c T \cup \{\conv(\polytope{S}\cup\{\b x\}) ~:~ \polytope{S}\in \c T\left|_{\c F}\right.\}$$ is a triangulation of the one-point extension $\conv(\pol\cup\{\b x\})$.

\item\label{item:thm:One_point_triangulation2} If $\c T$ is regular (respectively flag), then $\widetilde{\c T}$ is also regular (respectively flag).
\end{compactenum}
\end{theorem}

\begin{proof}
(1) Note that  $\pol = \bigcup_{ \simplex \in \c T}  \simplex$ for $\c T$ a triangulation of $\pol$.
Then by definition of $\c F_{\b x}$, we have that $\conv(\pol \cup \{\b x\}) = \pol \cup \bigcup_{ \polytope{F}\in\c F } \conv(\polytope{F} \cup \b x)$.
For any $\polytope{S}\in \c T\left|_{\c F_{\b x}}\right. $, we have that $\conv (\polytope{S} \cup \{\b x\})$ is a simplex, specifically, $\conv (\polytope{S} \cup \{\b x\}) = \Pyr(\polytope{S}, \b x)$.
Hence, by definition of visible facets, $\widetilde{\c T}$ defines a triangulation of $\conv(\pol \cup \{\b x\})$ with $\c T\subseteq\widetilde{\c T}$. Indeed, 
\[
\conv(\pol \cup \{\b x\}) = \bigcup_{\simplex\in \c T} \simplex ~~ \cup~~ \bigcup_{\polytope{S}\in \c T\left|_{\c F}\right.} \conv(\polytope{S} \cup \{\b x\}).
\]

(2) Let $\widetilde{\pol}=\conv(\simplex \cup \{\b x\})$. 
First we prove $\widetilde{\c T}$ is a regular triangulation. 
Pick a height vector $\b \omega \in \R^{V(\c T)}$ defined on the vertices $V(\c T)$ of the triangulation $\c T$, and consider $\b \omega' \in \R^{V(\c T)\cup\{\b x\}}$ defined by $\omega'_{\b a}~=~\omega_{\b a}$ for any $\b a\in \pol \cap V(\c T) $, and $\omega'_{\b x}$ taken arbitrarily large.
We set \linebreak $\pol_{\b \omega} \coloneqq \conv \left\{ 
\begin{pmatrix}
\b a \\
\omega_{\b a}
\end{pmatrix}
: \b a \in V(\c T)
\right\}$, and $\widetilde{\pol}_{\b \omega} \coloneqq \conv \left\{ 
\begin{pmatrix}
\b a \\
\omega_{\b a}
\end{pmatrix}
: \b a \in V(\c T) \cup \{\b x\}
\right\}$.
Pick  a lower facet $\polytope{G}$ of  $\pol_{\b \omega}$, then $\polytope{G}$ defines a half-space containing $\pol_{\b \omega}$ and supporting $\polytope{G}$.
As $\omega_{\b x}'$ is arbitrarily large (and $\polytope{G}$ is a lower facet), the point $\begin{pmatrix}
\b x \\
\omega'_{\b x}
\end{pmatrix}$ lies in the interior of this half-space.
Hence $\polytope{G}$ is also a facet of $\widetilde{\pol}_{\b \omega}$.
Besides, by definition of visibility, the other lower facets of $\widetilde{\pol}_{\b \omega}$ are formed by the convex hull of $\begin{pmatrix} \b x \\ \omega'_{\b x} \end{pmatrix}$ together with a facet of $\conv\left\{ 
\begin{pmatrix}
\b a \\
\omega'_{\b a}
\end{pmatrix}
: \b a \in \polytope{F} \cap V(\c T)
\right\}$ for some $\polytope{F}\in \c F$.
Thus, the triangulation $\widetilde{\c T}$ is regular.

It remains to prove that $\widetilde{\c T}$ is a flag triangulation of $\widetilde{\pol}$ if $\c T$ is flag. 
For this aim, let $\Delta_{\widetilde{\c T}}$ be the simplicial complex, associated to $\widetilde{\c T}$, and, similarly, define $\Delta_{\c T}$. By definition of $\widetilde{T}$, a minimal non-face of $\Delta_{\widetilde{\c T}}$ is either a minimal non-face of $\Delta_{c T}$ or it consists of  $\b x$ and exactly one vertex not lying in any visible facet in $\c F$. In particular, all minimal non-faces have cardinality $2$, and hence $\widetilde{\c T}$ is flag.
\end{proof}

%
%

%

%
%

%
%
%

%

%

%
%

%
%

%

%

%

%

%

%
%
%

%
%
%

%
%
%
%
%
%
%

%

    %
    
    %
    %
    
%
%

%
%

%
%
%
%
%
%
%
%
%
%
%
%
%
%
%
%
%
%
%
%
%
%
%
%

%
%
%
%
%
%
%
%
%
%
%
%

%
%

%
%
%
%
%
%
%
%
%
%
%

%

%

%
%
%
%
%
%

%
%
%
%
%
%
%
%
%

%
%
%
%

%
%
%
%

%

%

\section{Ehrhart non-positivity}\label{Section:Ehrhart_nonpositive}

Recall that for a sequence $\b s = (s_1,\dots, s_n)\in\N^n$, the \defn{$\b s$-lecture hall simplex} $\Pn$ is defined as
$$\mathcal{P}_n^{\b s} = \left\{\b x\in \R^n ~:~ 0 \leq \frac{x_1}{s_1} \leq \dots \leq \frac{x_n}{s_n}\leq 1\right\}.$$

The goal of this section is to provide numerous cases of Ehrhart non-positivity.
In particular, in \Cref{The:Negativity} we show that the $\b s$-lecture hall simplex is not Ehrhart positive for $\b s = (a,\dots, a,a+1)$ for any $a$ large enough (provided that the length of $\b s$ is at least $5$; see \Cref{thm:positivity} for the Ehrhart positivity when the length is $\leq 4$).

We now describe our central strategy to prove \Cref{The:Negativity}.
First we split the $h^\ast$-polynomial of $\Pn$ into a sum of the (shifted) $h^\ast$-polynomials of two simpler $\b s$-lecture hall simplices, see \Cref{prop:recursion_star}.
This relation allows us to express the Ehrhart polynomial of $\Pn$ (for $\b s=(a,\dots,a,a+1)$) as the sum of
$(i)$ the Ehrhart polynomial associated with the $a$\textsuperscript{th} dilation of the standard simplex, and $(ii)$ a correction term $\mathcal{R}(t)$ defined in \Cref{Thm:Conjecture_prove}.
Looking at the coefficient of both polynomials in front of $t^{n-4}$ as $a\to +\infty$, we prove that the coefficient in $(ii)$ grows faster than the one of $(i)$, and tends to $-\infty$.
Consequently, for any $a$ large enough, the coefficient in front of $t^{n-4}$ in the Ehrhart polynomial $\mathcal{L}_{\Pn}(t)$ is also negative, where $\b s=(a,\dots,a,a+1)$ and $n\geq 5$.

We start by introducing some notation. 
In the following, for a polynomial $P(t)$ in the variable~$t$ we use $[t^\ell]P(t)$ to denote the coefficient in front of $t^\ell$.
Given a sequence $\b s' = (s_1,\dots, s_{n-1})$ of positive integers of length $n-1$ and a positive integer $s_n$, we write $(\b s', s_{n})$ for the sequence $\b s=(s_1,\dots,s_{n-1},s_n)$ of length $n$.
With this notation at hand we can provide the desired splitting of the $h^\ast$-polynomial of an $\b s$-lecture hall simplex:

\begin{proposition}\label{prop:recursion_star}
Let $\b s = (s_1,\dots, s_n, s_n + 1)$ and $ {\b s'} = (s_1,\dots,s_{n-1}, s_n)$.
Then\footnote{Note that $\b s$ is of length $n+1$, whereas $\b s'$ is of length $n$.}
$$h_{\mathcal{P}_{n+1}^{\b s}}^*(t) = h_{\mathcal{P}_{n+1}^{(\b s',s_n)}}^*(t)  \,+\,  t\cdot  h_{\mathcal{P}_{n}^{\b s'}}^*(t).$$
\end{proposition}

\begin{proof}
Fix $\b s$.
By \Cref{eq:Eulerian-h-star}, we need to show that 
$$E_{n + 1}^{\b s}(t) = E_{n + 1}^{(\b s',s_{n})}(t) + t\cdot E_{n}^{\b s'}(t).$$

By definition, we have $E_{n+1}^{\b s}(t) = \sum_{\b e\in I_{n + 1}^{\b s}}t^{\asc(e)}$. 

Let $I_{n + 1}^{s_n}\coloneqq  \{ (e_1,\dots, e_{n + 1})~:~ 0\leq e_i < s_i ~ \text{ for } 1\leq i\leq n\text{ and } ~ e_{n + 1} = s_{n} \}$ and recall that 
$I_{n + 1}^{(\b s',s_n)} = \{ (e_1,\dots, e_{n + 1})~:~ 0\leq e_i < s_i ~ \text{ and } ~ e_{n + 1} < s_{n} \}$. 
Using these two disjoint sets, we can decompose the set of inversion sequences for $\b s$ as follows:
$I_{n + 1}^{\b s} =I_{n + 1}^{s_n}\,\sqcup\, I_{n + 1}^{(\b s',s_n)}$.

Since inversion sequences in $I_{n + 1}^{s_n}$ always have an ascent at the position $n$, regardless of the ascents at the other positions, we have 
\[
\sum_{\b e\in I_{n + 1}^{s_n}} t^{\asc(e)} = \sum_{\b e\in I_{n}^{\b s'}} t^{\asc(e) + 1}=t\cdot E_n^{\b s'}(t).
\]
This yields the desired formula:
\[
    E_{n + 1}^{\b s}(t) = \sum_{\b e\in I_{n + 1}^{\b s}} t^{\asc(e)} = \sum_{\b e\in I_{n + 1}^{(\b s',s_n)}} t^{\asc(e)} + \sum_{\b e\in I_{n}^{\b s'}} t^{\asc(e) + 1} 
    = E_{n + 1}^{(\b s',s_n)}(t) + t\cdot E_{n}^{\b s'}(t)
\qedhere\]
\end{proof}

We are now ready to present the first significant result of this section.

\begin{theorem}\label{thm:positivity}
For all $\b s = (a, \dots, a, a+1)$ of length $n\leq 4$ with $a\geq 1$, the $\b s$-lecture hall simplex $\mathcal{P}_n^{\b s}$ is Ehrhart-positive.
\end{theorem}

\begin{proof}
Recall that the constant coefficient of the Ehrhart polynomial is always $1$, and the highest and second highest coefficient of the Ehrhart polynomial account for the volume and the relative surface area, respectively. 
These coefficients are non-negative, and the claim follows for $n\leq 2$.

Now let $n\in \{3,4\}$ and let $\b a=(a,\dots,a)$ denote a constant sequence (of length $n$ or $n-1$). Moreover, let $\polytope{F}$ be the facet of $\mathcal{P}_n^{\b a}$ given by the hyperplane $\{\b x\in \R^n~:~x_{n-1}=x_n\}$.
The key observation is that this facet is unimodularly equivalent to $\c P_{n-1}^{\b a}$. 
Furthermore, note that $\mathcal{P}_n^{\b s}$ is (the translation by $\b e_n$ of) the union of $\mathcal{P}_n^{\b a}$ and the pyramid over the facet $\polytope{F}$ with apex $-\b e_n$, that is, $\mathcal{P}_n^{\b s} = \b e_n + \bigl(\mathcal{P}_n^{\b a}\, \cup\, \Pyr(\polytope{F},-\b e_n)\bigr)$. 
Now taking $\polytope{Q} = \mathcal{P}_n^{\b a}\, \cup\, \Pyr(\polytope{F},-\b e_n)$, 
\begin{align*}
    \mathcal{L}_{\polytope{Q}}(t) &= \mathcal{L}_{\mathcal{P}_n^{\b a}}(t) + \mathcal{L}_{\Pyr(\polytope{F}, -\b e_n)}(t) - \mathcal{L}_{\mathcal{P}_n^{\b a}\, \cap \, \Pyr(\polytope{F}, -\b e_n)}(t)\\
     &= \mathcal{L}_{\mathcal{P}_n^{\b a}}(t) + \mathcal{L}_{\Pyr(\polytope{F}, -\b e_n)}(t-1).
\end{align*}

Now using the fact that the facet $\polytope{F}$ is unimodularly equivalent to $P_{n-1}^{\b a}$, we have that
\[
\mathcal{L}_{\mathcal{P}_n^{\b s}}(t) = \mathcal{L}_{\mathcal{P}_n^{\b a}}(t) + \mathcal{L}_{\polytope{Pyr}(\mathcal{P}_{n-1}^{\b a},-\b e_{n-1})}(t - 1).
\]
    Since $\mathcal{P}_n^{\b a}$ is unimodular equivalent to the $a$\textsuperscript{th} dilation of the standard $n$-simplex, $\mathcal{L}_{\mathcal{P}_n^{\b a}}(t) = \binom{at + n}{n}$. 
    Using this and \eqref{Ehrhart: Pyramid}, 
    \begin{equation}\label{eq:crucial}
    \mathcal{L}_{\mathcal{P}_n^{\b s}}(t) = \binom{at + n}{n} + \sum_{\ell=0}^{t-1}\binom{a\ell+n}{n}.
\end{equation}    
    For $n = 3$, the only critical coefficient of $\mathcal{L}_{\Pn}$ is the coefficient of $t$ and it easily follows from \eqref{eq:crucial} that %
$[t]\mathcal{L}_{\mathcal{P}_3^{\b s}}(t) = \frac{1}{12}a^2 + \frac{13}{12}a + 1 \geq 0$ for $a\geq 1$.

Similarly, for $n = 4$, the critical coefficients appear at $t$ and $t^2$ and it is easily seen from \eqref{eq:crucial} that $[t]\mathcal{L}_{\mathcal{P}_4^{\b s}}(t) = \frac{1}{6}a^2 + \frac{7}{6}a + 1\geq 0$ and  $[t^2]\mathcal{L}_{\mathcal{P}_4^{\b s}}(t) = \frac{1}{24}a^3 + \frac{23}{24}a^2 + \frac{11}{12}a\geq 0$ for $a\geq 1$.
\end{proof}
Though one might hope that \Cref{thm:positivity} holds for sequences $\b s = (a,\dots,a,a+1)$ of length $n\geq 5$, the following example contradicts this belief.

\begin{example}\label{ex:negative}
It can be checked computationally, for $a\leq 15$ and $\b s=(a,a,a,a,a+1)$, that the simplex $\Pn$ is Ehrhart positive. However, for $a=16$, we have a negative linear coefficient in
\[
\mathcal{L}_{\Pn}(t) = \frac{139264}{15}t^5 + \frac{21760}{3}t^4 + \frac{9248}{3}t^3 + \frac{2210}{3}t^2 \textcolor{red}{- \frac{119}{15}}t + 1.\]

There is computational evidence that led us to expect that this behavior continues to hold for $a > 16$.
In \Cref{The:Negativity}, we prove that this is indeed the case for all $n\geq 5$: there is a certain threshold, such that for all $a$ bigger than this threshold, the polynomial $\c L_{\Pn}(t)$ for $\b s = (a, \dots, a, a+1)$ has a negative coefficient on $t^{n-4}$.
This motivates the following conjecture:
\end{example}

\begin{conjecture}
For $n\geq 5$, we conjecture that if $\c P_n^{(a, \dots, a, a+1)}$ is Ehrhart non positive for some $a\in \N$, then for all $b\geq a$ the $\b s$-lecture hall simplex $\c P_n^{(b, \dots, b, b + 1)}$ is also Ehrhart non positive.
\end{conjecture}

Building towards this conjecture, for small values of $n$, we computed the smallest $a$ such that $[t^{n-4}]\c L_{\c P_n^{(a, \dots, a, a+1)}}(t)<0$, which we name $\beta (n)$:
\begin{center}\begin{tabular}{c|llll}
$n$ & $5$ & $6$ & $7$ & $8$ \\ \hline
$\beta(n)$ & $16$ & $19$ & $23$ & $27$
\end{tabular}\end{center}

The main result of this section, \ie \Cref{thm:A}, generalizes \Cref{ex:negative} as follows.

\begin{theorem}\label{The:Negativity}
For any $n\geq 5$, there exists $\alpha(n)$ such that for all $a\geq \alpha(n)$ and $\b s = (a,\dots, a, a + 1)$ of length $n$, the simplex $\Pn$ is not Ehrhart positive.
More precisely, as $a\to \infty$,
\[
[t^{n-4}]\cL_{\Pn}(t) \sim -\,\frac{1}{720(n-4)!}\, a^{n-1}.
\]
\end{theorem}

We want to emphasize that there are also other coefficients in $\mathcal{L}_{\Pn}(t)$ that are negative, but not covered by our theorem.

While it may seem natural to  prove the above theorem using a similar strategy as in \Cref{thm:positivity}, this approach works well only for small $n$. 
The reason is that, when computing the coefficients in \Cref{eq:crucial}, the complexity of each step increases significantly, eventually making the expressions impossible to handle.

Thus, we will prove \Cref{The:Negativity} using an alternative approach.  
However, the proof requires several preliminary results.
First, recall that for $1 \leq k \leq n-1$, the \defn{elementary symmetric polynomial $e_k(Y_1, \dots, Y_{n-1})$} of degree $k$ in the variables $Y_1, \dots, Y_{n-1}$ is defined by
\[
e_k(Y_1, \dots, Y_{n-1}) =  \sum_{1\leq j_1<\dots <j_k \leq n-1} Y_{j_1}\dots Y_{j_k}.
\]
\begin{lemma}\label{lem:Elementary_degree}
For $k\geq 0$, the $k$\textsuperscript{th} elementary symmetric polynomial $e_k(1 - \ell,\, 2 - \ell,\, \dots,\, n - 1 - \ell)$ is a polynomial of degree $2k$ in the variable $n$, and of degree $k$ in the variable $\ell$.
\end{lemma}

\begin{proof}
Throughout this proof, we set $F_{n-1,\ell} \coloneqq \{1 - \ell, 2 - \ell, \dots, n - 1 - \ell\}$.  
It is well-known (see \cite[page 20]{macdonald1998symmetric}) that elementary symmetric functions satisfy the recurrence relation:
$$
e_k(F_{n - 1,\ell}) = (n - 1 - \ell) \, e_{k-1}(F_{n - 2,\ell}) \,+\, e_k(F_{n - 2,\ell}).
$$

Using the above recurrence relation, we argue recursively that $e_k(F_{n - 1,\ell})$ is a polynomial of degree $2k$ in the variable $n$ and degree $k$ in the variable $\ell$.
First, observe that we have the evaluation $e_1(F_{n - 1, \ell}) = \sum_{j = 1 - \ell}^{n - 1 - \ell} j = \frac{1}{2} \bigl((n - 1 - \ell)(n - \ell) - (\ell - 1)\ell \bigr),$ 
which is a polynomial of degree $2$ in the variable~$n$ and degree $1$ in the variable~$\ell$.
Next, note that (since $F_{a,\ell}$ is non-empty only if $a \geq 1$)
$$
e_k(F_{n - 1, \ell}) = \sum_{i = 1}^{n-2} e_{k - 1}\bigl(F_{n - 1 - i, \ell}\bigr) (n - i - \ell).
$$

According to Faulhaber's formula, a sum over $i$ from $0$ to $n$ of polynomials of degree $K$ in the variable $i$ yields a polynomial of degree $K + 1$ in the variable $n$.
Consequently, if we suppose that $e_{k-1}(F_{n - 1 -i , \ell})$ is a polynomial of degree $2(k - 1)$ in the variables $n$ and $i$, and degree $k-1$ in the variable $\ell$, then the product $e_{k - 1}(F_{n -1-i, \ell})(n - i - \ell)$ is a polynomial of degree $2k-1$ in the variables $n$ and $i$ and degree $k$ in the variable $\ell$.
Finally, Faulhaber's formula implies that $e_k(F_{n - 1, \ell})$ is a polynomial of degree $2k$ in the variable $n$ and degree $k$ in the variable $ \ell $.
This concludes the recursive argument, and the proof is complete.
\end{proof}

To prove \Cref{The:Negativity} we will need the following identity for a particular evaluation of $e_3$.

\begin{corollary}\label{lemma:elementary_symmetric}
For $\ell, n\geq 0$, the $3^{\text{rd}}$ elementary symmetric polynomial in $n-1$ variables satisfies 
\[
e_3(1 - \ell, \dots, n - 1 - \ell) = \frac{1}{48}(n - 3)(n - 2)(n - 1)(n - 2\ell)(n^2 - 4n\ell + 4\ell^2 - n).
\]
\end{corollary}

\begin{proof}
By \Cref{lem:Elementary_degree}, we know that $e_3(F_{n-1,\ell})$ is a polynomial of degree $6$ in $n$ and degree $3$ in $\ell$.
As the right-hand side shares the same degrees, the claim follows by verifying that both polynomials agree for all $0\leq n \leq 6$ and all $0\leq \ell \leq 3$. 
\end{proof}

The other crucial ingredient to prove \Cref{The:Negativity} is the following identity involving the Eulerian numbers (defined in \Cref{remark:Eulerian_numbers}).

\begin{lemma}\label{eq:Eulerian_property}
Let $A(n, \ell)$ be the Eulerian numbers.
Then, for $n\geq 5$ we have
$$\sum_{\ell = 1}^{n - 1}\frac{A(n - 1 , \ell - 1)}{(n-1)!}\, \ell(n - 2\ell)(n^2 - 4n\ell + 4\ell^2-n) = \frac{n}{15}.$$
\end{lemma}

The key idea is to differentiate the exponential generating function of the Eulerian numbers with respect to the parameters $x$ and $t$.
Using these partial derivatives, we can build up (the series of) the product of Eulerian numbers $A(n, \ell)$ by any polynomial in the variables $n$ and $\ell$.
In order to comply with technical difficulties, we need to integrate our series against $x$ and to withdraw some of the initial terms, but in essence, we are simply manipulating polynomials.

\begin{proof}[Proof of \Cref{eq:Eulerian_property}]
To simplify notation, let $\omega(n - 1,\ell - 1) = \ell(n - 2\ell)(n^2 - 4n\ell + 4\ell^2-n)$.
For positive integers $k$ and $\ell$, we use  $\ell^{\underline{k}} = \ell(\ell-1)\dots(\ell-k+1)$ to denote the falling factorial of $\ell$.

In the following, we study the generating series $G(x , t) \coloneqq \sum_{n\geq0}\sum_{\ell=0}^n A(n,\ell)\omega(n,\ell) \, t^\ell\frac{x^{n+1}}{(n+1)!}$, and its evaluation $G(x , 1)$ at $t=1$. 

The claim will follow from proving that the truncation $\tilde{G}(x , 1)\coloneqq \sum_{n\geq4}\sum_{\ell=0}^n A(n,\ell)\omega(n,\ell)\, \frac{x^{n+1}}{(n+1)!}$ equals the generating series $\sum_{n \geq 5} \frac{1}{15}x^n$.

To compute $G(x,t)$, we first consider the exponential generating series  $E(x,t)$ for the Eulerian numbers (see for instance \cite[Theorem 1.6]{Eulerian_book}), given by 

\begin{equation}\label{eq:Eulerian}
 E(x,t) = \sum_{ 0\leq \ell \leq n}A(n,\ell) \, t^\ell\frac{x^n}{n!} = \frac{t - 1}{t-\exp\bigl((t - 1)x\bigr)}.
 \end{equation}

In order to go from $E(x,t)$ to $G(x,t)$, the crucial idea is to express $G(x,t)$ as linear combination of appropriate partial derivatives of $E(x,t)$.
We start by computing the derivatives. For fixed positive integers $a,b \geq 0$, 

\begin{equation}\label{eq:derivatives}
t^ax^b\,\frac{\partial^{a+b} E(x, t)}{\partial t^a\, \partial x^b}
=  \sum_{n\geq0}\sum_{\ell=0}^n A(n,\ell)\, \ell^{\underline{a}}n^{\underline{b}}\, t^\ell\frac{x^n}{n!}.
\end{equation}

Since $\bigr(\ell^{\underline{a}}n^{\underline{b}} ~ : ~ 0\leq a\leq 4,\, 0\leq  b \leq 3 \big)$ is a basis for the vector space of polynomials of degree at most $4$ in $\ell$ and at most $3$ in $n$, we can rewrite $\omega(n,\ell )$ as follows:

\begin{equation}\label{eq:lincomb}
\omega(n, \ell) = \sum_{\substack{0 \leq a \leq 4 \\ 0 \leq b \leq 3}} \lambda_{a, b} \, \ell^{\underline{a}}n^{\underline{b}} 
\end{equation}
(see \Cref{Anex} for the explicit values of the coefficients $\lambda_{a,b}\in \Q$).

Combining \eqref{eq:derivatives} and \eqref{eq:lincomb} we get

 $$\frac{\partial G}{\partial x}(x , t) = 
 \sum_{n\geq0}\sum_{\ell=0}^n A(n,\ell)\omega(n,\ell) \, t^\ell\frac{x^n}{n!}
 =\sum_{\substack{0 \leq a \leq 4 \\ 0 \leq b \leq 3}}\lambda_{a,b}\, t^a x^b\frac{\partial^{a+b} E(x, t)}{\partial t^a \partial x^b}.$$

 Using the expression for $E(x,t)$ on the right-hand side of \eqref{eq:Eulerian} one can compute the derivatives above  explicitly, yielding an expression for $\frac{\partial G}{\partial x}(x , t)$ as a rational function, which allows for evaluation at $t = 1$. This way, we obtain

$$\frac{\partial G}{\partial x}(x , 1)
= -\frac{4x^3 - 20x^2 + 30x - 15}{15(x-1)^2}.$$

Integrating both sides with respect to $x$, we find 

$$ G(x , 1) = -\frac{x^3(x^2-5x+5)}{15(x-1)}
$$
and discarding terms of degree $\leq 4$ allows us to conclude the proof:
$$
\tilde{G}(x,1)=\frac{x^5}{15(1-x)} = \sum_{n \geq 5} \frac{1}{15}x^n. \qedhere$$
\end{proof}

\begin{theorem}\label{Thm:Conjecture_prove}

Let $a\in \N$ be a positive integer and let $\b a^{(n-1)} = (a,\dots, a)$ be the constant sequence of length $n-1$. 
Let $h_i^\ast(\b a^{(n-1)})$ denote the $i$\textsuperscript{th} entry of the $h^\ast$-vector of $\mathcal{P}_{n-1}^{\b a^{(n-1)}}$. We define $\mathdefn{\mathcal{R}(t)} \coloneqq \sum_{i=0}^{n-1} h_i^\ast(\b a^{(n-1)})\binom{t + n - i - 1}{n}$. 
Then, as $a\rightarrow \infty$,
\begin{equation*}
[t^{n - 4}]\mathcal{R}(t) \sim \frac{-a^{n-1}}{720(n-4)!}.
\end{equation*}

\end{theorem}
\begin{proof}
We start by rewriting the polynomial $\mathcal{R}(t)$ as follows:
\begin{align*}
\mathcal{R}(t) &=\sum_{i=0}^{n-1} h_i^{\ast}(\b a^{(n-1)})\binom{t  + n-i - 1}{n-1}\cdot\frac{(t-i)}{n}\\
&= \frac{t}{n}\sum_{i=0}^{n-1} h_i^{\ast}(\b a^{(n-1)})\binom{t   + n-i - 1}{n-1} \,\,-\,\, \frac{1}{n}\sum_{i=0}^{n-1} h_i^{\ast}(\b a^{(n-1)})\, i \,\binom{t + n - i - 1}{n-1}\\
&= \underbrace{\frac{t}{n} \mathcal{L}_{\mathcal{P}_{n-1}^{\b a^{(n-1)}}}(t)}_{g(t)} \,\,-\,\, \underbrace{\frac{1}{n}\sum_{i=0}^{n-1} h_i^{\ast}(\b a^{(n-1)}) \, i \,\binom{t  + n-i - 1}{n-1}}_{f(t)},
\end{align*}
where the last equality follows from \Cref{Lemma:Beck_book}.  
Since $\mathcal{P}_{n-1}^{\b a^{(n-1)}}$ is unimodular equivalent to the $a$\textsuperscript{th} dilation of the standard $(n-1)$-simplex, we get  $g(t)= \frac{t}{n}\binom{at + n - 1}{n - 1}$ (see e.g., \cite{beck2015computing}) which implies that there exists $m_{n}\in\Q^+$ such that $[t^{n - 4}]g(t) \sim a^{n - 5}m_{n}$  when $a\rightarrow +\infty$.
To determine the asymptotics of $[t^{n - 4}]\mathcal{R}(t) = [t^{n - 4}]g(t) - [t^{n - 4}]f(t)$ it is therefore sufficient to show, when $a\to+\infty$, that $[t^{n - 4}]f(t) \sim \frac{a^{n-1}}{720(n - 4)!}$. 
Using that $\binom{t+n-i-1}{n-1} = \frac{1}{(n-1)!}\bigl((t+n-i-1)(t+n-i-2)\dots t \dots (t-i)\bigr)$ is a polynomial of degree $n-1$ in $t$, we conclude from the definition of $f(t)$ and \Cref{lemma:elementary_symmetric}:
\begin{align*}
[t^{n - 4}]f(t) &= \frac{1}{n!}\sum_{i=0}^{n-1} h_i^{\ast}(\b a^{(n-1)})\, i \,e_3(1 - i, \dots, n - 1 - i) \\
&= \frac{1}{48n(n-4)!}\sum_{i=0}^{n-1} h_i^{\ast}(\b a^{(n-1)})\, i (n - 2i)(n^2 - 4ni + 4i^2 - n).
\end{align*}

Now, observe that:
$$
h_i^{\ast}(\b a^{(n-1)}) = \#\left \{(x_1,\dots, x_n)\in\N^n ~:~ x_j < a ~\text{ and }~ \sum_{\ell = 1} ^n x_\ell = ia\right \} = \mathcal{L}_{\simplex_{n,i}}(a)
$$
where $\mathdefn{\simplex_{n, i}} \coloneqq \{\b x\in \R^n ~:~ x_1+\dots+x_n = i\}$ is the hypersimplex with parameters $n$ and $i$.
Since $\mathcal{L}_{\simplex_{n,i}}(t)$ is a polynomial in $t$ with leading term $\vol(\simplex_{n,i}) $, it follows (see e.g., \cite{stanley1977eulerian,MR3558056}) that, 
\[
h_i^{\ast}(\b a^{(n-1)}) \sim  \vol(\simplex_{n,i})\cdot a^{n-1} = \frac{A(n - 1, i - 1)}{(n-1)!}a^{n-1},
\]
for large $a$ (see \Cref{remark:Eulerian_numbers} for the definition of the Eulerian numbers $A(n - 1, i - 1)$).
Thus, when $a \to +\infty$,
\begin{align*}
[t^{n - 4}]f(t) &\sim \frac{a^{n-1}}{48n(n-4)!}\sum_{i=0}^{n-1} \frac{A(n-1,i-1)}{(n-1)!} i (n - 2i)(n^2 - 4ni + 4i^2 - n) &&\\
&\sim \frac{a^{n-1}}{48n(n-4)!}\frac{n}{15} && \text{by \Cref{eq:Eulerian_property}}\\
&\sim \frac{a^{n-1}}{720(n-4)!}. &&
\qedhere\end{align*}
\end{proof}

We are now ready to prove \Cref{The:Negativity}.

\begin{proof}[Proof of \Cref{The:Negativity}]
Let $\b s = (a,\dots, a, a+1)$ be a sequence of length $n$ with $a \geq 1$, and let $\b a^{(n)} = (a,\dots, a)$ denote the constant subsequence of length~$n$.  
Then, by \Cref{prop:recursion_star}, the $h^\ast$-polynomial of $\mathcal{P}_n^{\b s}$ can be expressed as follows:

\begin{align*}
    h^\ast(t) & = h_0^\ast(\b s) + h_1^\ast (\b s) t + \ldots + h_n^\ast (\b s) t^n \\
              & = h_0^\ast(\b a^{(n)}) + \Bigl( h_1^\ast(\b a^{(n)}) + h_0^\ast(\b a^{(n-1)})\Bigr) t + \ldots + \Bigl(h_n^\ast(\b a^{(n)}) + h_{n-1}^\ast(\b a^{(n-1)})  \Bigr) t^n.
\end{align*}

Therefore, by applying \Cref{Lemma:Beck_book}, we deduce  the following expression:

$$\begin{array}{rcccl}
\mathcal{L}_{\Pn}(t) &=& h_0^{\ast}(\b s)\binom{t + n}{n}  +  \dots +h_n^{\ast}(\b s)\binom{t}{n} & \\
&=& \sum_{i = 0}^n h_i^{\ast}(\b a^{(n)}) \binom{t + n - i}{n} &+& \sum_{i = 0}^{n - 1} h_i^{\ast}(\b a^{(n-1)}) \binom{t + n - 1 - i}{n}\\
&=& \binom{at + n}{n} &+& \mathcal{R}(t), 
\end{array}$$
where in the last step we use that, by \Cref{Lemma:Beck_book}, the first sum equals $\mathcal{L}_{\mathcal{P}_n^{\b a^{(n)}}}(t)$, which, as in the proof of \Cref{thm:positivity}, equals $\binom{at+n}{n}$.

Therefore, by \Cref{Thm:Conjecture_prove}, we have that $[t^{n - 4}]\mathcal{L}_{\Pn}(t) =  \frac{-a^{n-1}}{720(n-4)!} + \omega(n,a)$ where, for a fixed $n$, we have $\deg(\omega(n,a)) < n - 1$ (for the degree as a polynomial in the variable $a$).
Hence, there exist $\alpha(n)$ such that for all $a\geq\alpha(n)$, we have $[t^{n - 4}]\mathcal{L}_{\Pn}(t) < 0$.
\end{proof}

\section{Regular, flag and unimodular triangulation}\label{Section:Regular_Unimodular}

The prevalence of (flag, regular) unimodular triangulations in discrete geometry motivates the search for their existence for specific classes of polytopes.
In particular, for $\b s$-lecture hall simplices, Hibi, Olsen and Tsuchiya  formulated the  following conjecture, that inspires this section.

\begin{conjecture}[{\cite[Conjecture 5.2]{hibi2016}}]\label{conj:Unimo}
For any $\b s$, $\Pn$ admits a unimodular triangulation.
\end{conjecture}

Some progress has been made towards proving this conjecture.
Foremost, \cite[Theorem~3.3 \& Remark~3.4]{hibi2016} proves that $\Pn$ admits a unimodular triangulation if each entry of $\b s$ is a positive integral multiple or a divisor of the previous one, \ie for each $i$ there exists $k_i\geq 1$ such that  $s_{i+1} = k_i\cdot s_i$ or $s_{i+1} = \frac{s_i}{k_i}$.
Additionally, it was shown by Br\"anden and Solus in \cite[Corollary~3.5]{SolusAndBradden} that for sequences $\b s$ with  $0\leq s_{i+1} - s_i \leq 1$ for all $i$, the toric ideals of $\Pn$ has a square-free Gr\"obner basis: this implies that $\Pn$ admits a regular unimodular triangulation also in this case.

The goal of this section is to provide further evidence for  \Cref{conj:Unimo} by showing that it remains true for sequences that arise as kind of combinations of the aforementioned ones. 
This generalizes both results from \cite{hibi2016} and \cite{SolusAndBradden}.
We are going to present some preliminary results, that will be necessary for the proper development of this section.

\begin{lemma}\label{lemma:s_lecture_one_point_extension}
Let $\b s = (s_1, \dots, s_{n-1}, s_n)$ and $\b s' = (s_1, \dots, s_{n-1}, s_n + 1)$.
Then $\c P_n^{\b s'}$ is the one-point extension of (the translation by $\b e_n$ of) $\c P_n^{\b s}$ by the point $\b 0$.
\end{lemma}
\begin{proof}

Note that $\Pn + \b e_n$ has the same set of vertices as $\mathcal{P}_n^{\b s'}$, except for the point $\b 0$.
Hence, $\conv\bigl( (\Pn + \b e_n) \cup \{\b 0\} \bigr)$, contain the same set of vertices that $\mathcal{P}_n^{\b s'}$ and the result follows.
\end{proof}

\begin{lemma}\label{lem:Unimodular_triangulation}
Let $\b s = (s_1, \ldots, s_n)$ and $\b s' = (s_1, \ldots, s_n + 1)$ such that $s_{n} = k \cdot s_{n-1}$ for some positive integer $k$. 
Let $\polytope{F}$ be the facet of $\Pn$ whose supporting hyperplane is given by $\{\b x\in \R^n~ :~ \frac{x_{n-1}}{s_{n-1}}=\frac{x_n}{s_n}\}$.
Let $\mathcal{T}$ be a triangulation of $\Pn$, and let $\mathcal{T}\left|_{\polytope{F}}\right.=\{\simplex\cap \polytope{F}~:~ \simplex\in \mathcal{T}\}$ be the induced triangulation on~$\polytope{F}$. 
Then
\[
\widetilde{\mathcal{T}} = \{\simplex+\b e_n~:~\simplex\in\mathcal{T}\} \cup \{\conv(\simplex, -\b e_n)+\b e_n~ :~ \simplex\in\mathcal{T}\left|_{\polytope{F}}\right. \}
\]
is a triangulation of $\mathcal{P}_n^{\b{s}'}$. 
Moreover, if $\mathcal{T}$ is unimodular, then so is $\widetilde{\mathcal{T}}$.
\end{lemma}

\begin{proof}
The result that $\widetilde{\c T}$ be a triangulation of $\mathcal{P}_n^{\b s'}$ is a direct consequence of \Cref{thm:One_point_triangulation}(1) and \Cref{lemma:s_lecture_one_point_extension}.
Therefore, it is only necessary to verify that $\widetilde{\c T}$ is unimodular.
Suppose $s_n = k\cdot s_{n-1}$ for a positive integer $k$ and let $\b r \coloneqq (s_1,\ldots,s_{n-1})$. In this case, the facet $\polytope{F}$ is given as
\[
\polytope{F}=\{(x_1,\ldots,x_{n-1}, \frac{s_n}{s_{n-1}}x_{n-1})~:~ (x_1,\ldots,x_{n-1})\in \polytope{P}_{n-1}^{\b r}\}
\]
and as $\frac{s_n}{s_{n-1}}= k\in \mathbb{N}$, it follows that $\polytope{F}$ is unimodular equivalent to $\polytope{P}_{n-1}^{\b r}$.
Now assume that $\mathcal{T}$ is unimodular. Since $\mathcal{T}\left|_{\polytope{F}}\right.$ is the restriction of $\mathcal{T}$ to $\polytope{F}$, it is unimodular as well.
Hence, we know that the relative volume of $\polytope{F}$ (with respect to its ambient lattice) equals the volume of $\polytope{P}_{n-1}^{\b r}$, which is $ \prod_{i=1}^{n-1}s_i$, and it follows that the number of maximal simplices in the triangulation $\widetilde{\mathcal{T}}$ is
\[
|\mathcal{T}|+\mathrm{nvol}(\c P_n^{\b r})=\mathrm{nvol}(\Pn)+\prod_{i=1}^{n-1}s_i=\prod_{i=1}^{n-1}s_i\cdot (ks_{n-1}) + \prod_{i=1}^{n-1}s_i=\prod_{i=1}^{n-1} s_i \cdot s'_n.
\]

As this number is equal to the relative volume of $\c P_n^{\b s'}$, we conclude that $\widetilde{\mathcal{T}}$ is unimodular.
\end{proof}

\begin{theorem}
\label{Thm:unimodular-tri}
Let $\b s = (s_1,\dots, s_{n-1})$, and $s_n = k\cdot s_{n-1} + \varepsilon$ with $k\geq 0$ and $\varepsilon\in\{0, 1\}$ (with $s_n\neq 0$).
If $\mathcal{P}_{n - 1}^{\b s}$ admits a unimodular, flag, and regular triangulation, then so does $\mathcal{P}_{n}^{(\b s,s_n)}$.
\end{theorem}

\begin{proof}
First assume $k\geq 1$. 
For the case $\varepsilon = 0$, a unimodular triangulation was constructed in \cite[Thm. 3.3]{hibi2016} using chimney polytopes. It follows from \cite[Thm 2.8 \& Cor 2.9]{MR4277268} that this triangulation is also regular and flag.

Now let  $\varepsilon = 1$ and $k\geq 1$. 
Let $\c T$ be a unimodular, flag and regular triangulation of $\mathcal{P}_n^{(\b s, ks_{n-1})}$, which exists due to the already handled case $\epsilon=0$. 
Then, by \Cref{lem:Unimodular_triangulation},
\[
\widetilde{\mathcal{T}} = \{\simplex+\b e_n~:~\simplex\in\mathcal{T}\} \cup \{\conv(\simplex, -\b e_n)+\b e_n~ :~ \simplex\in\mathcal{T}_{\polytope{F}} \}
\]
is a unimodular triangulation of $\mathcal{P}_n^{\b s'}$.
Moreover, by \Cref{thm:One_point_triangulation}(2) and \Cref{lemma:s_lecture_one_point_extension},  since $\c T$ is flag and regular, so is $\widetilde{\c T}$.

 Finally, assume $k=0$. Since $s_n\neq 0$, we have $s_n=1$, \ie $\varepsilon = 1$. 
Observe that $\mathcal{P}_{n}^{(\b s,s_n)}$ is the pyramid over $\mathcal{P}_{n-1}^{\b s}$ embedded in $\R^{n}$ at height $x_{n} = 1$  with apex $\b 0$.
This implies that if  $\mathcal{T}$ is a flag, regular, unimodular triangulation of $\mathcal{P}_{n-1}^{\b s}$, then 
$\widetilde{\mathcal{T}} = \{\Pyr(\simplex, \b 0) ~ : ~ \simplex\in\mathcal{T} \}$ is a flag, regular, unimodular triangulation of $\mathcal{P}_{n}^{\b{s'}}$.
\end{proof}

\begin{remark}\label{rem:reversedSequence}
For $\b s = (s_1, \dots, s_n)$, let $\b s^{\mathrm{rev}} = (s_n,\dots, s_1)$ be the reversed sequence.
The simplices $\c P_n^{\b s^{\mathrm{rev}}}$ and $\Pn$ are unimodular equivalent: applying central symmetry to $\Pn$, then translating by the vector $(s_1, \ldots, s_n)$, then permuting (reversing) the coordinates, one gets $\c P_n^{\b s^{\mathrm{rev}}}$. %
Consequently, the polytope $\c P_n^{\b s^{\mathrm{rev}}}$ admits a flag, regular, unimodular triangulation if and only if $\Pn$ admits one.
Note that this argument was used in \cite[Remark 3.4]{hibi2016} to conclude that sequences $\b s$ satisfying $s_{i+1}=\frac{s_{i}}{k_i}$ with integers $k_i\geq 1$, for all $1\leq i\leq n-1$, give rise to $\b s$-lecture hall simplices having a unimodular triangulation.

With this idea, one can directly prove the ``reversed'' version of \Cref{Thm:unimodular-tri} that is:
For $\b s = (s_1, \dots, s_{n-1})$ and $s_0 = k\cdot s_1 + \varepsilon$ with $k\geq 0$ and $\varepsilon\in \{0, 1\}$, if $\c P_{n-1}^{\b s}$ admits a flag, regular, unimodular triangulation, then $\c P_n^{(s_0, \b s)}$ admits a flag, regular, unimodular triangulation.
\end{remark}

Combining \Cref{Thm:unimodular-tri,rem:reversedSequence}, we get the following direct consequence.

\begin{corollary}\label{cor:extended}
Let $\b s = (s_1,\dots, s_n)$.
If there exists $1\leq m\leq n$ such that 
\begin{enumerate}
    \item[(1)] for all $1\leq i\leq m-1$ either $s_i = 1$ or there is $k_i\geq 1$ and $\epsilon_i\in\{0,1\}$ with $s_{i+1} =  \frac{s_{i}-\epsilon_i}{k_i}$; and
    \item[(2)] for all $m\leq i\leq n-1$ either $s_{i+1} = 1$ or there is $k_i\geq 1$ and $\epsilon_i\in\{0,1\}$ with $s_{i+1} = k_i\cdot s_{i} + \epsilon_i$;
\end{enumerate}
then $\mathcal{P}_{n}^{\b s}$ has a flag, regular, unimodular triangulation.
\end{corollary}

\begin{proof}
Fix $\b s$ satisfying the above conditions, and let $\b r = (s_1, \dots, s_m)^{\mathrm{rev}}=(s_m,\ldots,s_1)$.
By (1), we have for all $1\leq i\leq m-1$ that $r_{i+1} = k_i\cdot r_i + \varepsilon_i$ with $k_i\geq 0$ and $\varepsilon_i \in\{0, 1\}$.
Hence, applying \Cref{Thm:unimodular-tri} repetitively ensures that $\c P_m^{\b r}$ admits a flag, regular, unimodular triangulation (as $\c P_1^{(r_1)}$ admits one).
Consequently, by the reversing property of \Cref{rem:reversedSequence}, the simplex $\c P_m^{(s_1, \dots, s_m)}$ admits a flag, regular, unimodular triangulation too.
Applying \Cref{Thm:unimodular-tri} repetitively for $m\leq i\leq n-1$, condition (2) guarante     es that $\Pn$ admits a flag, regular, unimodular triangulation.
\end{proof}

\begin{remark}
We want to emphasize that all the triangulations discussed in \Cref{Thm:unimodular-tri,rem:reversedSequence} are constructed explicitly (contrarily to \cite{LiuAndSolus2019}) by induction, assuming that the triangulation of $\c P_n^{\b s}$ is given explicitly (see \cref{fig:Lemma_pyramid}).
This remains true for the triangulations coming from the more general sequence in \Cref{cor:extended}.
\begin{figure}[H] 
    \centering
    \begin{tikzpicture}[scale=0.5][very thick,opacity=1]








\begin{scope}[shift={(-5, 0)}, fill opacity= 1]

    \fill[fill = green!40,  opacity=0.6] (0,3) -- (0,11) -- (6,11) ;

    \draw (0, 0) node[below left]{\scriptsize{$(0,0)$}};
    \draw (0, 3.3) node[below left]{\scriptsize{$(0,1)$}};
    \draw (0, 11.5) node[below left]{\scriptsize{$(0,s_n + 1)$}};
    \draw (10 , 11.5) node[below left]{\scriptsize{$(s_{n - 1}, s_n + 1)$}};

    \draw[dotted]  (0 , 3) -- (6 , 11) ;
    \draw[black, very thick](0 , 3) -- (0, 11) ;
    \draw[black, dotted](0 , 3) -- (0, 0) ;
    \draw[red, thick](0 , 3) -- (6 , 11) ;
    \draw[black, very thick](0 , 11) -- (6 , 11) ;

    \draw (0.45 , 0.5 ) node[below left]{$\bullet$};
    \draw (0.45 , 11.4 ) node[below left]{$\bullet$};
    \draw (6.4 , 11.4 ) node[below left]{$\bullet$};
    \draw (7 , 10) node[below]{\textcolor{red}{$\polytope{F} + \b e_n$}};
    \draw (-2 , 7.4) node[below]{\textcolor{green}{$ \Pn + \,\b e_n$}};
    \draw (1.5 , 8.3) node[below]{$ \mathcal{T}$};
    \draw (0.45, 3.5) node[below left]{$\bullet$};
    \draw (2.05 , 5.55) node[below left]{$\bullet$};
    \draw (4.5 , 8.8) node[below left]{$\bullet$};
    \draw (2.5 , 6.7) node[below left]{$\cdot$};
    \draw (2.9 , 7.2) node[below left]{$\cdot$};
    \draw (3.3 , 7.7 ) node[below left]{$\cdot$};

\end{scope}

\begin{scope}[shift={(11.8, 0)}, fill opacity= 1]

    \fill[fill=blue!30,  opacity=0.6] (0,0) -- (0,3) -- (6,11) ; 
    \fill[fill = green!40,  opacity=0.6] (0,3) -- (0,11) -- (6,11) ;

    \draw (0, 0) node[below left]{\scriptsize{$(0,0)$}};
    \draw (0, 3.3) node[below left]{\scriptsize{$(0,1)$}};
    \draw (0, 11.5) node[below left]{\scriptsize{$(0,s_n + 1)$}};
    \draw (10 , 11.5) node[below left]{\scriptsize{$(s_{n - 1}, s_n + 1)$}};

    \draw[red, thick]  (0 , 3) -- (6 , 11) ;
    \draw[black, very thick](0 , 0) -- (0, 11) ;
    \draw[black, very thick](0 , 0) -- (6 , 11) ;
    \draw[black, very thick](0 , 11) -- (6 , 11) ;
    \draw[dotted]  (0 , 0) -- (1.6 , 5) ;
    \draw[dotted]  (0 , 0) -- (4.2 , 8.6) ;
    \draw (2 , 5 ) node[below left]{$\cdot$};
    \draw (2.4 , 5 ) node[below left]{$\cdot$};
    \draw (2.2 , 5 ) node[below left]{$\cdot$};
    \draw (0.45 , 0.5 ) node[below left]{$\bullet$};
    \draw (0.45 , 11.4 ) node[below left]{$\bullet$};
    \draw (6.4 , 11.4 ) node[below left]{$\bullet$};
    \draw (0.45, 3.5) node[below left]{$\bullet$};
    \draw (2.05 , 5.55) node[below left]{$\bullet$};
    \draw (4.5 , 8.8) node[below left]{$\bullet$};
    \draw (2.5 , 6.7) node[below left]{$\cdot$};
    \draw (2.9 , 7.2) node[below left]{$\cdot$};
    \draw (3.3 , 7.7 ) node[below left]{$\cdot$};
    \draw (1.5 , 8.3) node[below]{$ \mathcal{T}$};
    \draw (7 , 10) node[below]{\textcolor{red}{$\polytope{F} + \b e_n$}};
    \draw (5.5, 4.5) node[below]{\textcolor{blue!40}{$\Pyr(\polytope{F},\b -e_n) + \b e_n$}};

    \draw (-2 , 7.4) node[below]{\textcolor{green}{$ \Pn + \,\b e_n$}};

\end{scope}

\end{tikzpicture}
    \caption{Triangulation of $\mathcal{P}_n^{\b s'}$ with $\b s' = (s_1, \ldots, s_n + 1)$.}
    \label{fig:Lemma_pyramid}
\end{figure}
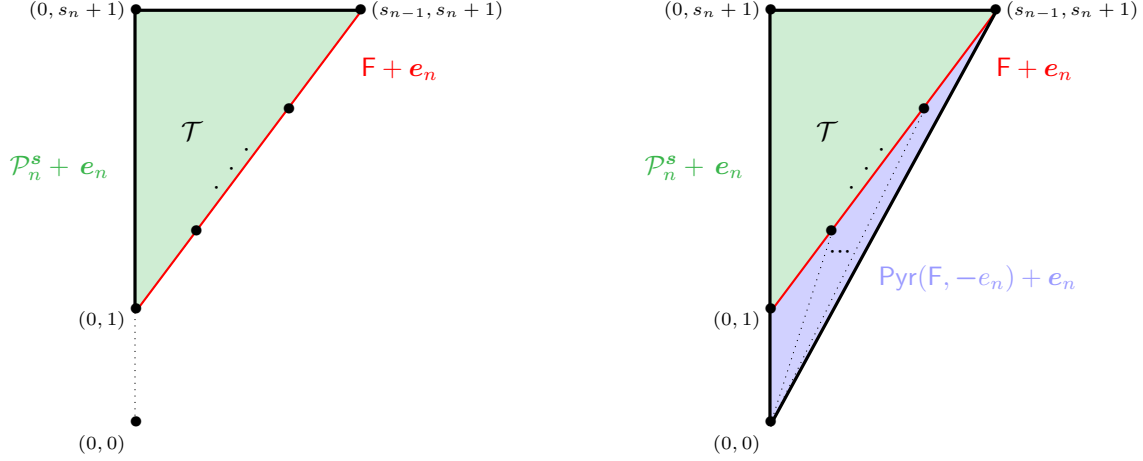  

\end{remark}

\begin{remark}
By \Cref{cor:extended}, for a given $\b s$, we can run a test in linear time that, if positive, tells us that $\Pn$ admits a flag, regular, unimodular triangulation; but can give us false-negative: namely, testing condition (1) \& (2) for all $1\leq i\leq n-1$ (and concluding on the existence of $m$).
\end{remark}

\begin{example}
We give explicit $\b s$ for which $\Pn$ admits a flag, regular, unimodular triangulation:
\begin{enumerate}
\item Consider the sequence $\b s= (1,2,4,5,10,11,12,24)$. %
Observe that each term can be expressed in the form $s_{i+1} = k_is_i + \epsilon$ where $\epsilon_i\in\{0,1\}$ and $k_i \geq 1$. 
Thus, by  \Cref{cor:extended}(1), the simplex $\mathcal{P}_{8}^{\b s}$ admits a flag, regular, unimodular triangulation.
\item The sequence $\b s' = (17, 8, 4, 1,10, 3, 6, 13)$ falls into the class of sequences described in \Cref{cor:extended}(1)\&(2), with $m = 6$.
Hence, the simplex  $\mathcal{P}_{8}^{\b s'}$ admits a flag, regular, unimodular triangulation.
One should note that this sequence $\b s$ is not increasing nor decreasing.
\end{enumerate}
\end{example}

\begin{example}\label{example:wherever}
For $\b s = (3, 5, 2)$, a computer experiments yields a flag, regular, unimodular triangulation for~$\Pn$ (see \Cref{Anex}).
Using \Cref{Thm:unimodular-tri,rem:reversedSequence}, we can create an infinite family of new sequences $\b s'$ for which $\c P_n^{\b s'}$ admits a flag, regular, unimodular triangulation:
it is enough that $s'_{m+1} = s_1, \dots, s'_{m+3} = s_3$, and that $(s'_1, \dots, s'_{m+1})$ satisfies condition (1) of \Cref{cor:extended}, while $(s'_{m+3}, \dots, s'_n)$ satisfies condition (2) of \Cref{cor:extended}.

Up to our knowledge, no previous theorem (including \Cref{cor:extended} it self) allowed to prove that $\c P_n^{\b s'}$ admits a flag, regular, unimodular triangulation, for such aforementioned $\b s'$.

\end{example}

One might wonder whether \Cref{Thm:unimodular-tri} extends to cases where $s_n = ks_{n-1} + \epsilon$ with $\epsilon \geq 2$.  
Our proof crucially relies on the fact that when $\epsilon \in \{0,1\}$, exactly one new lattice point appears in the region $\polytope{R}_2$ (see \Cref{fig:remark_final}).  
On the opposite, for $\epsilon \geq 2$, the region $\polytope{R}_3$ in \Cref{fig:remark_final} may contain multiple new lattice points.  
In such cases, the one-point extension construction fails to yield a unimodular triangulation, as some lattice points would not be vertices of the extended triangulation.  
It remains unclear whether a unimodular refinement can still be constructed.
Motivated by this, we leave the following question open:

\begin{Question}
For a sequence $\b s = (s_1, \dots, s_{n-1})$ and $s_n = ks_{n-1} + \epsilon$ with $\epsilon \geq 2$, can a flag, regular, unimodular triangulation of $\c P_n^{(\b s, s_n)}$ be explicitly constructed from one of $\Pn$?
\end{Question}

\begin{figure}[ht]
\usetikzlibrary{patterns,patterns.meta}
\centering
\begin{tikzpicture}[scale=0.85]

\begin{scope}[shift={(8, 0)}]
\draw[fill = green!20] (0, 2) -- (0, 8) -- (3,8) -- cycle;
\draw[fill = blue!20] (0, 1) -- (0, 2) -- (3,8) -- cycle;
\draw[fill = violet!20] (0, 0) -- (0, 1) -- (3,8) -- cycle;

\foreach \x in {0,...,3}{
    \foreach \y in {0,...,8}{
        \pgfmathtruncatemacro{\comp}{int(3*\y - 8*\x)}
        \ifnum\comp> -1
            \draw(\x , \y) node{$\bullet$};
        \fi
    }
}

\draw (-0.1 , 5) node[below left]{$\textcolor{green}{\polytope{R}_1}$};
\draw (-0.1 , 1.7) node[below left]{$\textcolor{blue!50}{\polytope{R}_2}$};
\draw (-0.1 , 0.7) node[below left]{$\textcolor{violet!50}{\polytope{R}_3}$};


\end{scope}

\begin{scope}[shift={(0, 0)}]
\draw[fill = green!20] (0, 2) -- (0, 8) -- (3,8) -- cycle;
\draw[fill = blue!20] (0, 1) -- (0, 2) -- (3,8) -- cycle;

\foreach \x in {0,...,3}{
    \foreach \y in {1,...,8}{ 
        \pgfmathtruncatemacro{\comp}{int(3*\y - 8*\x)}
        \ifnum\comp> -1
            \draw(\x , \y) node{$\bullet$};
        \fi
    }
}
\fill[white] (1,3) circle (3pt);
\draw (-0.1 , 5) node[below left]{$\textcolor{green}{\polytope{R}_1}$};
\draw (-0.1 , 1.7) node[below left]{$\textcolor{blue!50}{\polytope{R}_2}$};

\end{scope}

\end{tikzpicture}
\caption{(Left) For $\b s = (3,\, 2\cdot 3+1)$, only $\textcolor{blue!50}{\polytope{R}_2}$ appears, with a single point added.
(Right) The subpolytopes $\textcolor{blue!50}{\polytope{R}_2}$ and $\textcolor{violet!50}{\polytope{R}_3}$ for $\b s = (3,\, 2\cdot 3+2)$.}
\label{fig:remark_final}
\end{figure}
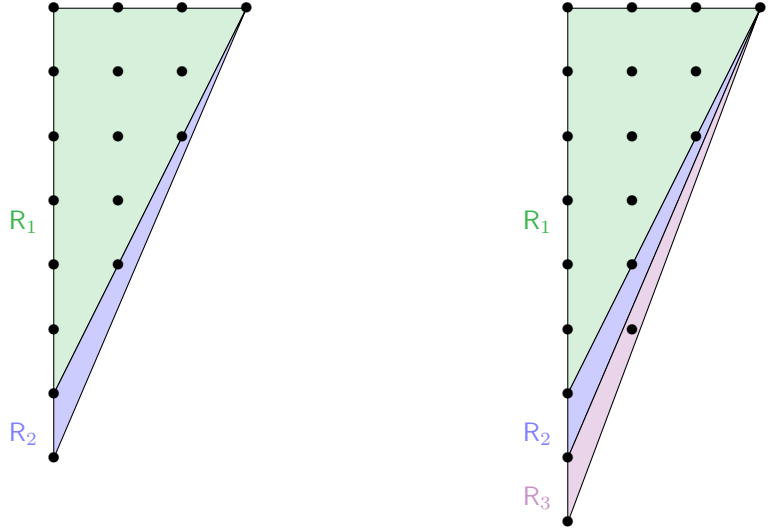

We extend \Cref{cor:extended} by applying some known results on flag, regular, unimodular triangulations of joins and dilations of polytopes.
We refer to \cite{MR4277268} for the detail of these tools.

\begin{definition}
    Let $\pol\subset \R^n $ and $\pol'\subset \R^{n'} $ two polytopes.
    Then define the \defn{join} of $\pol $ and $\pol'$ as $\mathdefn{\pol \ast \pol'} \coloneqq \conv(\{ (\b x , 0, \b 0) \, : \,  \b x\in \pol \} \cup  \{ ( \b 0 , 1, \b x') \, : \,  \b x'\in \pol' \} ) \subset \R^{n + n' + 1}$.
\end{definition}

\begin{proposition}\label{prop:join}
    Let $\b s$ and $\b s'$ be two integer sequences, and let $\b r = (\b s,\, 1, \, \b s')$.
    Then the polytope $\c P_{n + n' + 1}^{\b r}$ is (unimodular equivalent to) the join of the polytopes $\c P_n^{\b s}$ and $\c P_{n'}^{\b s'}$.
\end{proposition}

\begin{proof}
Let $\b s = (s_1, \dots, s_n)$, $\b s' = (s'_1, \dots, s_{n'}')$.
Let $E_{i,j}$ be the matrix that be $0$ in everywhere, except for its entry in row $i$, column $j$, which is $1$.
We take the matrix $T = I_{n + n' + 1} - \sum_{j = n' + 2}^{n+n'+1} E_{n' + 1, j}$.
Note that $T$ is unimodular because it is upper-triangular with $1$s on the diagonal.
We have

$$\hspace{-0.75 cm} \begin{pmatrix}
    0 &  &        &   &       &  & &  & 0 & s_1 \\
     &   &        &  &         &  & &  & s_2 & s_2 \\
    &    &        &   &         &  & & \iddots & \vdots  & \vdots  \\
    &    &        &  &         &  & s_n  &\cdots & s_n & s_n \\
    &    &        &   &         & 1 & 1  &\cdots & 1 & 1 \\
    &    &        &  & s'_{1'} & s'_{1'} & s'_{1'} & \cdots & s'_{1'} & s'_{1'} \\ 
    &    &        & s_2'  & s'_{2'} & s'_{2'} & s'_{2'} & \cdots & s'_{2'} & s'_{2'} \\ 
    &    &    \iddots    &  \vdots & \vdots  & \vdots & \vdots & \vdots & \vdots & \vdots \\ 
  0 &   s_{n'}' & \dots & s_{n'}'  & s_{n'}' & \cdots & s'_{n'} & s'_{n'} & \cdots &  s'_{n'}
\end{pmatrix} ~\cdot~ T = 
\begin{pmatrix}
    0 &  &        &   &       &  & &  & 0 & s_1 \\
     &   &        &  &         &  & &  & s_2 & s_2 \\
    &    &        &   &         &  & & \iddots & \vdots  & \vdots  \\
    &    &        &  &         &  & s_n  &\cdots & s_n & s_n \\
    &    &        &   &         & 1 & 1  &\cdots & 1 & 1 \\
    &    &        &  & s'_{1'} & 0 & 0 & \cdots & 0 & 0 \\ 
    &    &        & s_2'  & s'_{2'} & 0 & 0 & \cdots & 0 & 0 \\ 
    &    &    \iddots    &  \vdots & \vdots  & \vdots & \vdots & \vdots & \vdots & \vdots \\ 
  0 &   s_{n'}' & \dots & s_{n'}'  & s_{n'}' & \cdots & 0 & 0  & \cdots & 0
\end{pmatrix}. $$

The columns of the right-hand-side matrix are the vertices of $\c P_n^{\b s'} \ast \c P_{n'}^{\b s'}$.
This shows that $\c P_{n+n'+1}^{\b r}$ is unimodular equivalent to $\c P_n^{\b s'} \ast \c P_{n'}^{\b s'}$.
\end{proof}

\begin{corollary}\label{coro:join}
    Let $\b s$ and $\b s'$ be two integer sequences, and let $\b r = (\b s,\, 1, \, \b s')$.
    If $\c P_n^{\b s}$ and $\c P_n^{\b s'}$ admit flag, regular, unimodular triangulations, then $\c P_n^{\b r}$ also admits a flag, regular, unimodular triangulation for $\b r = (\b s,\, 1, \, \b s')$.
\end{corollary}

\begin{proof}
By \cite[Section 2.3.2]{MR4277268}, if two polytopes admit flag, regular, unimodular triangulations, then their join also admits one.
By \Cref{prop:join}, the polytope $\c P_n^{\b r}$ is unimodular equivalent to the join of $\c P_n^{\b s}$ and $\c P_n^{\b s'}$.
\end{proof}

\begin{remark}\label{Remark:dilation}
    We presented two operations building on concatenation to ensure the existence of (flag, regular) unimodular triangulations.
    Note that we can also dilate the sequence, namely denoting $\lambda \b s = (\lambda s_1, \cdots, \lambda s_n)$ for some integer $\lambda\geq 0$, we get $\c P_n^{\lambda \b s} = \lambda \c P_n^{\b s}$.
    As the latter only amounts to dilation, by \cite[Theorem 4.11]{MR4277268}, if $\c P_n^{\b s}$ admits a flag, regular, unimodular triangulation, then $ \c P_n^{\lambda \b s}$ also admits a flag, regular, unimodular triangulation.
\end{remark}

Finally, we present a result that connects \Cref{Section:Ehrhart_nonpositive,Section:Regular_Unimodular}.
To simplify notation, for positive integers $a$, $b$, $k_2$, and nonnegative integers $k_1$, $k_3$, we write $(\b 1^{k_1},\, a ,\, \b 1^{k_2},\, b,\, \b 1^{k_3})$ to denote the sequence $(\underbrace{1,\dots,1}_{k_1},\, a,\, \underbrace{1,\dots,1}_{k_2},\, b,\, \underbrace{1,\dots,1}_{k_3}).$

\begin{corollary}\label{Coro:Unimodular_Solus}
For $\b s = (\b 1^{k_1},\, a ,\, \b 1^{k_2},\, b,\, \b 1^{k_3})$, $\Pn$ admits a flag, regular, unimodular triangulation.
\end{corollary}

\begin{proof}
Note that for all $i$, we have that $s_{i+1}\in\{1\cdot s_i,\, a\cdot s_i,\, b\cdot s_i,\, 1\}$.
Thus, using \Cref{Thm:unimodular-tri} over each $1\leq i \leq n-1$, the result follows.
\end{proof}

\begin{corollary}\label{coro:constant_a}
For $\b s = (a, \ldots, a, a + 1)$, $\Pn$ admits a flag, regular, unimodular triangulation.
\end{corollary}

\begin{proof}
Note that for all $i$, we have that $s_{i+1}\in\{s_i,\, s_i + 1\}$.
Thus using \Cref{Thm:unimodular-tri} over each $1\leq i \leq n-1$ the result follows.
\end{proof}

Liu and Solus \cite[Theorem 4.3]{LiuAndSolus2019} proved that for $n = k_1+k_2+k_3+2 \geq 3$, there are values of $k_1$, $k_2$, $k_3$, $a$, $b$ such that the simplex $\mathcal{P}_{n}^{(\b 1^{k_1},\, a,\, \b 1^{k_2},\, b,\, \b 1^{k_3})}$ is not Ehrhart positive.
Consequently, according to \Cref{The:Negativity,Coro:Unimodular_Solus,coro:constant_a}, both sequences $\b s = (\b 1^{k_1},\, a ,\, \b 1^{k_2},\, b,\, \b 1^{k_3})$ and $\b s = (a, \ldots, a, a + 1)$ give rise to $\b s$-lecture hall simplices $\Pn$ which are Ehrhart non-positive but admit a flag, regular, unimodular triangulation.
This contrast is particularly striking, as it illustrates the richness of the combinatorial and geometric behavior of $\b s$-lecture hall simplices: a single sequence can give rise to polytopes with seemingly contrasting properties. 
Such phenomena underscore the importance of studying these simplices from multiple perspectives.

%
%

%

\appendix
\section{Numerical values}\label{Anex}

All the values below have been computed using SageMath \cite{Sage}.

\paragraph{$\lambda_{a, b}$ for the proof of \Cref{eq:Eulerian_property}}

Here are the values of $\lambda_{a,b}$ for $0\leq a \leq 4$ and $0\leq b \leq 3$:

\begin{center}
    \begin{tabular}{c|c c c c c}
       $b\backslash a$  & $0$ & $1$ & $2$ & $3$ & $4$ \\ \hline
       $0$  &      & $-48$ & $-132$ & $-68$ & $-8$\\
       $1$  &      & $40 $ &  $56$  &  $12$ & \\
       $2$  & $-1$ & $-13$ & $-6$   &       & \\
       $3$  & $1$  & $1$   &        &       &
    \end{tabular}
\end{center}

\paragraph{Explicit values of $\beta(n)$ as defined in \Cref{ex:negative}:}

\begin{center}\begin{tabular}{c|llll}
$n$ & $5$ & $6$ & $7$ & $8$ \\ \hline
$\beta(n)$ & $16$ & $19$ & $23$ & $27$
\end{tabular}\end{center}

\paragraph{Triangulation of $\c P_3^{(3,5,2)}$ from \Cref{example:wherever}}

The $18$ lattice points of $\c P_3^{(3,5,2)}$ are:
$$
\begin{array}{cccccc}
\b v_0 = (0, 0, 0) &  \b v_1 = (0, 0, 1) & \b v_2 = (0, 1, 1) 
&\b v_3 = (0, 2, 1) &  \b v_4 = (1, 2, 1) & \b v_5 =  (0, 0, 2) \\
\b v_6 = (0, 1, 2) &  \b v_7 = (0, 2, 2) & \b v_8 = (0, 3, 2) 
&\b v_9 = (0, 4, 2) &  \b v_{10} = (0, 5, 2) & \b v_{11} = (1, 2, 2) \\
\b v_{12} = (1, 3, 2) &  \b v_{13} = (1, 4, 2) & \b v_{14} = (1, 5, 2) 
&\b v_{15} = (2, 4, 2) &  \b v_{16} = (2, 5, 2) & \b v_{17} = (3, 5, 2) .
\end{array}
$$

The following list of $30$ subsets of indices $X$ gives a triangulation $\c T$ of $\c P_3^{(3, 5, 2)}$ whose simplices are $\conv\{\b v_i ~:~ i\in X\}$:
$$
\begin{array}{cccccc}
\{0,1,2,4\} & \{0,1,4,17\} & \{0,2,3,4\} & \{0,3,4,10\} & \{0,4,10,14\} & \{0,4,14,16\} \\
\{0,4,16,17\} & \{1,2,4,5\} & \{1,4,5,17\} & \{2,3,4,5\} & \{3,4,5,6\} & \{3,4,6,7\} \\
\{3,4,7,8\} & \{3,4,8,9\} & \{3,4,9,10\} & \{4,5,6,11\} & \{4,5,11,17\} & \{4,6,7,11\} \\
\{4,7,8,11\} & \{4,8,9,11\} & \{4,9,10,11\} & \{4,10,11,12\} & \{4,10,12,13\} & \{4,10,13,14\} \\
\{4,11,12,15\} & \{4,11,15,17\} & \{4,12,13,15\} & \{4,13,14,15\} & \{4,14,15,16\} & \{4,15,16,17\}.
\end{array}
$$

The flagness of $\c T$ can be checked by constructing the 1-skeleton of the triangulation and check directly that all its cliques are (faces of) simplices in the triangulation.

This triangulation $\c T$ is regular: using the following height vector $\b\omega$, a computer check ensures that the lower facets of $\conv\{(\b v_i, \omega_i) : i\in [18]\}$ are $\conv\{(\b v_i, \omega_i) : i\in X\}$ for $X\in \c T$:

$$\b \omega = (0,\, 1,\, 0,\, 0,\, 0,\, 3,\, 4,\, 6,\, 9,\, 13,\, 18,\, 7,\, 13,\, 20,\, 28,\, 31,\, 40,\, 53).$$

\bibliographystyle{alpha}
\bibliography{Bib.bib}

@article{Savagerealroots,
AUTHOR = {Savage, Carla D. and Visontai, Mirk\'o},
     TITLE = {The {$\bold{s}$}-{E}ulerian polynomials have only real roots},
   JOURNAL = {Trans. Amer. Math. Soc.},
  FJOURNAL = {Transactions of the American Mathematical Society},
    VOLUME = {367},
      YEAR = {2015},
    NUMBER = {2},
     PAGES = {1441--1466},
      ISSN = {0002-9947,1088-6850},
   MRCLASS = {05A05 (05A19 05A30 26C10)},
  MRNUMBER = {3280050},
MRREVIEWER = {Roberto\ Tauraso},
       DOI = {10.1090/S0002-9947-2014-06256-9},
       URL = {https://doi.org/10.1090/S0002-9947-2014-06256-9},
}

@incollection{Liu2019,
 AUTHOR = {Liu, Fu},
     TITLE = {On positivity of {E}hrhart polynomials},
 BOOKTITLE = {Recent trends in algebraic combinatorics},
    SERIES = {Assoc. Women Math. Ser.},
    VOLUME = {16},
     PAGES = {189--237},
 PUBLISHER = {Springer, Cham},
      YEAR = {2019},
      ISBN = {978-3-030-05141-9; 978-3-030-05140-2},
   MRCLASS = {05A15 (05A20 52B20)},
  MRNUMBER = {3969575},
MRREVIEWER = {Matthias\ Beck},
       DOI = {10.1007/978-3-030-05141-9\_6},
       URL = {https://doi.org/10.1007/978-3-030-05141-9_6},
}

@article{SAVAGE2012850,
AUTHOR = {Savage, Carla D. and Schuster, Michael J.},
TITLE = {Ehrhart series of lecture hall polytopes and {E}ulerian
              polynomials for inversion sequences},
   JOURNAL = {J. Combin. Theory Ser. A},
  FJOURNAL = {Journal of Combinatorial Theory. Series A},
    VOLUME = {119},
      YEAR = {2012},
    NUMBER = {4},
     PAGES = {850--870},
      ISSN = {0097-3165,1096-0899},
   MRCLASS = {05A17 (05A19 52B11)},
  MRNUMBER = {2881231},
MRREVIEWER = {Ae\ Ja\ Yee},
       DOI = {10.1016/j.jcta.2011.12.005},
       URL = {https://doi.org/10.1016/j.jcta.2011.12.005},
}

@book{beck2015computing,
   AUTHOR = {Beck, Matthias and Robins, Sinai},
     TITLE = {Computing the continuous discretely},
    SERIES = {Undergraduate Texts in Mathematics},
   EDITION = {Second},
      NOTE = {Integer-point enumeration in polyhedra,
              With illustrations by David Austin},
 PUBLISHER = {Springer, New York},
      YEAR = {2015},
     PAGES = {xx+285},
      ISBN = {978-1-4939-2968-9; 978-1-4939-2969-6},
   MRCLASS = {11P21 (05A15 05B15 11-02 11H06 52B05 52B20)},
  MRNUMBER = {3410115},
       DOI = {10.1007/978-1-4939-2969-6},
       URL = {https://doi.org/10.1007/978-1-4939-2969-6},
}

@article{GUSTAFSSON2020107169,
  AUTHOR = {Gustafsson, Nils and Solus, Liam},
     TITLE = {Derangements, {E}hrhart theory, and local {$h$}-polynomials},
   JOURNAL = {Adv. Math.},
  FJOURNAL = {Advances in Mathematics},
    VOLUME = {369},
      YEAR = {2020},
     PAGES = {107169, 35},
      ISSN = {0001-8708,1090-2082},
   MRCLASS = {52B20 (05A05 05A15 05E45 13P25)},
  MRNUMBER = {4091894},
MRREVIEWER = {Alex\ Fink},
       DOI = {10.1016/j.aim.2020.107169},
       URL = {https://doi.org/10.1016/j.aim.2020.107169},
}

@article{LiuAndSolus2019,
AUTHOR = {Liu, Fu and Solus, Liam},
     TITLE = {On the relationship between {E}hrhart unimodality and
              {E}hrhart positivity},
   JOURNAL = {Ann. Comb.},
  FJOURNAL = {Annals of Combinatorics},
    VOLUME = {23},
      YEAR = {2019},
    NUMBER = {2},
     PAGES = {347--365},
      ISSN = {0218-0006,0219-3094},
   MRCLASS = {52B20 (05A15 05A20)},
  MRNUMBER = {3962862},
MRREVIEWER = {Matthias\ Beck},
       DOI = {10.1007/s00026-019-00429-8},
       URL = {https://doi.org/10.1007/s00026-019-00429-8},
}

@incollection{Olsen-question,
AUTHOR = {Olsen, McCabe},
TITLE = {Polyhedral geometry for lecture hall partitions},
BOOKTITLE = {Algebraic and geometric combinatorics on lattice polytopes},
PAGES = {330--353},
PUBLISHER = {World Sci. Publ., Hackensack, NJ},
YEAR = {2019},
ISBN = {978-981-120-047-2},
MRCLASS = {05A17 (52B20)},
MRNUMBER = {3971703},
}

@article{hibi2016,
 AUTHOR = {Hibi, Takayuki and Olsen, McCabe and Tsuchiya, Akiyoshi},
     TITLE = {Gorenstein properties and integer decomposition properties of
              lecture hall polytopes},
   JOURNAL = {Mosc. Math. J.},
  FJOURNAL = {Moscow Mathematical Journal},
    VOLUME = {18},
      YEAR = {2018},
    NUMBER = {4},
     PAGES = {667--679},
      ISSN = {1609-3321,1609-4514},
   MRCLASS = {52B20 (05A15 05A20 05E40)},
  MRNUMBER = {3914109},
MRREVIEWER = {Robert\ Davis},
       DOI = {10.17323/1609-4514-2018-18-4-667-679},
       URL = {https://doi.org/10.17323/1609-4514-2018-18-4-667-679},
}

@article{SolusAndBradden,
  AUTHOR = {Br\"and\'en, Petter and Solus, Liam},
     TITLE = {Some algebraic properties of lecture hall polytopes},
   JOURNAL = {S\'em. Lothar. Combin.},
  FJOURNAL = {S\'eminaire Lotharingien de Combinatoire},
    VOLUME = {84B},
      YEAR = {2020},
     PAGES = {Art. 25, 12},
      ISSN = {1286-4889},
   MRCLASS = {52B20 (05A15 13P10)},
  MRNUMBER = {4138653},
}

@article{stanley1980decompositions,
AUTHOR = {Stanley, Richard P.},
     TITLE = {Decompositions of rational convex polytopes},
   JOURNAL = {Ann. Discrete Math.},
  FJOURNAL = {Annals of Discrete Mathematics},
    VOLUME = {6},
      YEAR = {1980},
     PAGES = {333--342},
   MRCLASS = {52A43},
  MRNUMBER = {593545},
MRREVIEWER = {P.\ McMullen},
}

@book {Eulerian_book,
AUTHOR = {Petersen, T. Kyle},
     TITLE = {Eulerian numbers},
    SERIES = {Birkh\"auser Advanced Texts: Basler Lehrb\"ucher.
              [Birkh\"auser Advanced Texts: Basel Textbooks]},
      NOTE = {With a foreword by Richard Stanley},
 PUBLISHER = {Birkh\"auser/Springer, New York},
      YEAR = {2015},
     PAGES = {xviii+456},
      ISBN = {978-1-4939-3090-6; 978-1-4939-3091-3},
   MRCLASS = {05-02 (05A15 05Exx 06A07 11B65 11B75 20F55)},
  MRNUMBER = {3408615},
MRREVIEWER = {Damir\ Yeliussizov},
       DOI = {10.1007/978-1-4939-3091-3},
       URL = {https://doi.org/10.1007/978-1-4939-3091-3},
}

@article{stanley1977eulerian,
title={Eulerian partitions of a unit hypercube},
author={Stanley, Richard P. },
journal={Higher Combinatorics (Proc. NATO Advanced Study Institute, Berlin, 1976)},
year={1977},
pages = {49},
url = {https://math.mit.edu/~rstan/pubs/pubfiles/34a.pdf}
}

@article {Ehrhart_poly,
    AUTHOR = {Ehrhart, Eug\`ene},
     TITLE = {Sur les poly\`edres rationnels homoth\'etiques \`a{} {$n$}\
              dimensions},
   JOURNAL = {C. R. Acad. Sci. Paris},
  FJOURNAL = {Comptes Rendus Hebdomadaires des S\'eances de l'Acad\'emie des
              Sciences},
    VOLUME = {254},
      YEAR = {1962},
     PAGES = {616--618},
      ISSN = {0001-4036},
   MRCLASS = {10.25 (52.10)},
  MRNUMBER = {130860},
}

@article {MR3558056,
  AUTHOR = {Liu, Gaku},
     TITLE = {Mixed volumes of hypersimplices},
   JOURNAL = {Electron. J. Combin.},
  FJOURNAL = {Electronic Journal of Combinatorics},
    VOLUME = {23},
      YEAR = {2016},
    NUMBER = {3},
     PAGES = {Paper 3.19, 19},
      ISSN = {1077-8926},
   MRCLASS = {52A39 (05A10)},
  MRNUMBER = {3558056},
MRREVIEWER = {Weidong\ Wang},
       DOI = {10.37236/5770},
       URL = {https://doi.org/10.37236/5770},
}

@article {FlorianAndOlsen,
    AUTHOR = {Kohl, Florian and Olsen, McCabe},
     TITLE = {Level algebras and {$s$}-lecture hall polytopes},
   JOURNAL = {Electron. J. Combin.},
  FJOURNAL = {Electronic Journal of Combinatorics},
    VOLUME = {27},
      YEAR = {2020},
    NUMBER = {3},
     PAGES = {Paper No. 3.50, 23},
      ISSN = {1077-8926},
   MRCLASS = {52B20 (05A17 13H10)},
  MRNUMBER = {4245163},
MRREVIEWER = {Mihai\ Cipu},
       DOI = {10.37236/8626},
       URL = {https://doi.org/10.37236/8626},
}

@book{macdonald1998symmetric,
  AUTHOR = {Macdonald, Ian G.},
     TITLE = {Symmetric functions and {H}all polynomials},
    SERIES = {Oxford Mathematical Monographs},
   EDITION = {Second},
      NOTE = {With contributions by A. Zelevinsky,
              Oxford Science Publications},
 PUBLISHER = {The Clarendon Press, Oxford University Press, New York},
      YEAR = {1995},
     PAGES = {x+475},
      ISBN = {0-19-853489-2},
   MRCLASS = {05E05 (05-02 20C30 20C33 20K01 33C80 33D80)},
  MRNUMBER = {1354144},
MRREVIEWER = {John\ R.\ Stembridge},
}

@manual {Sage,
	AUTHOR = {{S}age, {D}evelopers of},
	TITLE = {{S}age {M}athematics {S}oftware},
	NOTE = {\url{http://www.sagemath.org}},
	YEAR = {2016},
}

@article {MR4277268,
AUTHOR = {Haase, Christian and Paffenholz, Andreas and Piechnik, Lindsay C. and Santos, Francisco},
     TITLE = {Existence of unimodular triangulations---positive results},
   JOURNAL = {Mem. Amer. Math. Soc.},
  FJOURNAL = {Memoirs of the American Mathematical Society},
    VOLUME = {270},
      YEAR = {2021},
    NUMBER = {1321},
     PAGES = {v+83},
      ISSN = {0065-9266,1947-6221},
      ISBN = {978-1-4704-4716-8; 978-1-4704-6530-8},
   MRCLASS = {52B20 (13F20 13P10 14M25)},
  MRNUMBER = {4277268},
MRREVIEWER = {Margaret\ M.\ Bayer},
       DOI = {10.1090/memo/1321},
       URL = {https://doi.org/10.1090/memo/1321},
}
\label{sec:biblio}

\end{document}